\documentclass[11pt,reqno]{amsart}
\usepackage{amssymb,amsfonts,amsmath,bm}
\usepackage{hyperref}
\usepackage{psfrag,graphicx,color}
\usepackage{mathrsfs}
\usepackage{graphics}
\usepackage{subfigure}
\usepackage{enumerate,enumitem}
\usepackage{pgfplots}
\usepackage{cleveref}
\usepackage{cite}
\usepackage{diagbox}
\usepackage{hyperref}
\usepackage{changes} 

\usepackage{booktabs}
\usepackage{multirow,array,tabularx}

\usepackage[margin=1.1in, marginparwidth=1.1in, marginparsep=0.05in]{geometry}

\DeclareGraphicsExtensions{.eps}
\usepackage{amssymb,amsmath,graphicx,amsfonts,euscript,mathrsfs}
\usepackage{hyperref}
\usepackage{ulem}

\newtheorem{theorem}{Theorem}[section]
\newtheorem{lemma}{Lemma}[section]

\newtheorem{proposition}{Proposition}[section]
\newtheorem{remark}{Remark}[section]
\theoremstyle{definition}

\newtheorem{example}[theorem]{Example}

\usepackage{amsopn}

\DeclareMathOperator{\Span}{span}

\def\R{{\mathbb R}}

\numberwithin{equation}{section}
\numberwithin{figure}{section}

\newcounter{marnote}

\providecommand{\det}{\operatorname{det}} 

\providecommand{\det}{\operatorname{det}}

\newcommand{\diff}{{\rm d}}

\newcommand{\mcP}{\mathcal{P}}
\newcommand{\mcT}{\mathcal{T}}
\newcommand{\msC}{\mathsf{C}}
\newcommand{\vep}{\varepsilon}
\newcommand{\rd}{{\rm d}}
\newcommand{\dx}{{{\rm d}x}}
\newcommand{\set}[1]{\left\{ #1 \right\}}

\newcommand{\jump}[1]{\left[ #1 \right]}

\newcommand{\bbR}{\mathbb{R}}

\newcommand{\abs}[1]{\left\vert #1 \right\vert}
\newcommand{\bv}[1]{\big\vert #1 \big\vert}

\allowdisplaybreaks

\begin{document}

\title[FEM for the perfect conductivity and linear elastic]{High-order finite element method \\for perfect conductivity and linear elasticity \\with nearly touching inclusions} 

\author[]{Buyang Li}
\address{\hspace*{-12pt}Buyang Li: 
Department of Applied Mathematics, The Hong Kong Polytechnic University, Hung Hom, Hong Kong. {\it E-lail address}: {\tt buyang.li@polyu.edu.hk}}

\author[]{\,\,Haigang Li}
\address{\hspace*{-12pt}Haigang Li: School of Mathematical Sciences, Beijing Normal University, Laboratory of Mathematics and Complex Systems, Ministry of Education, Beijing 100875, People's Republic of China.
{\it Email address}: {\tt hgli@bnu.edu.cn}}

\author[]{\,\,Yonglin Li}
\address{\hspace*{-12pt}Yonglin Li:  
School of Mathematics and Statistics, Wuhan University, Wuhan 430072, People's Republic of China. {\it Email address}: {\tt yonglin.li@whu.edu.cn}}

\author[]{\,\,Peihao Zhang}
\address{\hspace*{-12pt}Peihao Zhang: School of Mathematics and Statistics, Henan University, Kaifeng 475004, People's Republic of China.
{\it Email address}: {\tt phzhang@henu.edu.cn}}

\subjclass[2010]{65N12; 
                 65N15; 
                 65N30; 
                 35B40  
                 }


\keywords{
  Perfect conductivity problem, 
  linear elasticity problem, 
  asymptotic estimates, 
  high-order finite element method, 
  error estimates
}

%
\begin{abstract}

In perfect conductivity and linear elasticity problems, the electric field and stress always become highly concentrated within narrow regions between adjacent perfect (rigid) inclusions, and blow up as the distance between inclusions approaches zero. The design of high-order numerical methods with rigorous error analysis for such concentration problems remains open. In this paper, we present the first high-order finite element method for solving these problems. Our approach is based on asymptotic estimates of high-order derivatives of solutions, employing a graded mesh and auxiliary basis functions specifically designed from these derivative estimates. We prove that the proposed method converges with an $H^1$-error bound of $O(h^p)$ and the error bound is independent of the distance (possibly approaching zero) between inclusions, where $h$ is the mesh size and $p$ is the degree of finite elements. Numerical examples in both two and three dimensions are presented to demonstrate the convergence rates of the numerical solutions. In particular, the blow-up behaviors of the gradients of solutions are demonstrated when the inclusions approach each other.
\end{abstract}

\maketitle



\section{Introduction}\label{sec:intr}
In high-contrast composite materials, when inclusions approach near-touching, physical fields such as stress or electric field  may become arbitrarily large within the narrow inter-inclusion regions. Precise characterization of this field concentration is critically important for both numerical simulation and partial differential equation theory. For fiber-reinforced composites containing densely packed elastic inclusions, stress concentration may induce material failure \cite{BASSL1999}. Similarly, conducting inclusions can generate significantly enhanced electric fields, enabling applications in subwavelength imaging and sensitive spectroscopy \cite{YuAmmari}.

For simplicity, we consider a bounded domain $D\subset\R^n$, $n\in\{2,3\}$, containing two smooth convex inclusions $D_1\subset D$ and $D_2\subset D$, which may be nearly touching, as shown in Figure \ref{figure:domain}. The voltage potential $u$ in  conductivity problem satisfies the boundary value problem
\begin{align}\label{PDE2}
  \begin{cases}
    \nabla\cdot\left(a_{k}(x) \nabla u\right)=0 & \text { in } D, \\ 
    u=\varphi& \text { on } \Gamma:=\partial D, 
  \end{cases}
\end{align} 
where $a_{k}(x)$ is the (renormalized) conductivity in different materials and defined by 
$$
  a_{k}(x)= 
  \begin{cases}
    k &\text{for}\,\, x\in D_{1} \cup D_{2} ,\\ 
    1 &\text{for}\,\, x\in \Omega:= D \backslash \overline{D_{1} \cup D_{2}}, 
  \end{cases}
$$
The equation \eqref{PDE2} is equivalent to finding the minimizer $u\in H^1(D)$ with $u|_{\Gamma} = \varphi$ of the energy functional 
\begin{align*}
  E_{k}[u]:=\frac{k}{2} \int_{D_{1} \cup D_{2}}|\nabla u|^{2}+\frac{1}{2} \int_{\Omega}|\nabla u|^{2} .
\end{align*}
It has been proved in \cite{AmmariKangLLZ2007,BV2000,LV2000,DongLo2019} that the electric field $\nabla u$ is bounded uniformly with respect to the distance $\varepsilon:={\rm dist}(D_1,D_2)$ provided that $k$ is away from $0$ and $\infty$.

However, when $k \to \infty$, the $L^{\infty}$-norm of $\nabla u$ becomes very large in the narrow region between two inclusions and blows up as $\varepsilon \to 0$. The limiting case $k\rightarrow\infty$ of this problem is known as the perfect conductivity problem, which characterizes the asymptotic behaviour with respect to $\varepsilon$. 
Investigating such asymptotic behaviour has been a long-standing interest in material science as well as computational mathematics; see \cite{AmmariBonnetierTV2015,AmmariKangLim2005,BASSL1999,BonnetierTriki2013,Dong2012,KLY,Keller1963,Kim2018,LV2000,Milton2002,McPhedran1986,MM1987,MPM1988} and the references therein. 
\begin{figure}[t]
  \centering
  \begin{tikzpicture}[scale=0.75, semithick]
    \draw (0, 0) circle(2.4) +(1.3, 1.5) node {\small$\Omega$};
    \draw (2.2, 0) node {\small$\Gamma$};
    \draw (0,  0.74) circle [x radius=1, y radius=0.7] node {\small$D_1$} +(1.25, 0) node {\small$\Gamma_1$};
    \draw (0, -0.74) circle [x radius=1, y radius=0.7] node {\small$D_2$} +(1.25, 0) node {\small$\Gamma_2$};
  \end{tikzpicture}
  \caption{The domain $\Omega=D\backslash \overline{D_1\cup D_2}$.}
  \label{figure:domain}
\end{figure}
To minimize the energy in this limiting case $k\rightarrow\infty$, the solution must satisfy $\nabla u = 0$ in each inclusion $D_j,\,j=1,2$.
By denoting 
$$
  H_{c}^{1}(\Omega)=\left\{v \in H^{1}(\Omega) : v = \text{constant on }\Gamma_j \text{ for }j=1,2\right\},
$$ 
the perfect conductivity problem can be written as finding $u \in H_c^1(\Omega)$ with $u|_\Gamma=\varphi$ that minimizes the energy functional
\begin{align*}
  E_{\infty}[u]:=\frac{1}{2} \int_{\Omega}|\nabla u|^{2}, 
\end{align*}
which can be rewritten by 
\begin{align}\label{PDE}
  \begin{cases}
    \Delta u=0 &\text {in}\,\, \Omega, \\ 
    u=c_{j} &\text {on}\,\, \Gamma_j:=\partial D_j,\, j=1,2, \\ 
    \int_{\Gamma_j} \partial_{\nu}u =0 & j=1,2, \\ 
    u=\varphi  &\text {on}\,\, \Gamma,
  \end{cases}
\end{align} 
where $c_1$ and $c_2$ are free constants to be determined by the condition 
$\int_{\Gamma_j} \partial_{\nu}u = 0$ with $j=1,2$, and $\partial_\nu$ denotes the outward normal derivative on $\Gamma_j$. 

The asymptotic estimates for the gradient of the solution to the perfect conductivity problem have been established for circular inclusions in \cite{KLY,KLY2,LWX2019} and for general convex smooth inclusions in \cite{Li20,LLY}. In particular, by denoting $x'=(x_1,\cdots,x_{n-1})$, if the boundaries of the two inclusions can be described (locally in a neighborhood of the origin) by the following two graphs: 
\begin{align}\label{local-geometry}
  {\tilde\Gamma_j}=\{x\in\R^n:x_n=\phi_j(x')\;\, \mbox{for}\;\,|x'| < 1\},\,\,\,j=1,2, 
\end{align}
with 
\begin{align}
  \phi_1(0)=-\phi_2(0)=\frac\varepsilon2,\; \nabla\phi_1(0)=\nabla\phi_2(0)=0 \;\;\mbox{and}\quad \rho\big(\nabla^2(\phi_1-\phi_2)(0)\big)\ge \lambda_0, \label{geo2}
\end{align} 
where $\rho(\cdot)$ denotes the smallest eigenvalue of a matrix and $\lambda_0$ is some fixed positive constant. The narrow domain between two inclusions is defined by
\begin{align}
  \Omega_{r}:=\{x=(x',x_n)\in\Omega: |x'|\le r,\; \phi_2(x') \leq x_n \leq \phi_1(x') \} \quad \mbox{for}\,\,\,r > 0. \label{def:Or}
\end{align} 
The vertical distance between $D_1$ and $D_2$ is given by $\delta(x'):=\phi_1(x')-\phi_2(x'),\ |x'| \leq r$. 
Applying Taylor's theorem, we have
\[
\delta(x') = \delta(0) + \nabla\delta(0)\cdot x' + \frac12 (x')^T \nabla^2\delta(0) x' + O(|x'|^3), \quad |x'|\le r.
\]
It follows from \eqref{geo2} that, there exists a constant $C_0>0$ such that
\begin{equation}\label{djx}
   C_0^{-1}(\varepsilon + |x'|^2)\leq \delta(x')\leq C_0(\varepsilon + |x'|^2).
\end{equation} 

Some integral equation methods and expansion methods have proven effective in computing solutions to conductivity and linear elasticity problems in composite materials under various settings; see the fast-multipole integral equation methods \cite{GM1994}, the fast-multipole iterative schemes \cite{Helsing1995,Helsing1996}, the method of images \cite{CG1997,CG1998}, the boundary integral method (for acoustic modelling of high-contrast media) \cite{whgbs22}, and the hybrid basis scheme \cite{CTD2016}, which specifically addresses the challenge posed by nearly touching disk-shaped inclusions in two dimensions. 
These methods mainly focus on the two-dimensional conductivity problem, involving either mildly close inclusions of general geometry or nearly touching disk-shaped inclusions. A spectral Galerkin approximation of an integral equation formulation was proposed in \cite{LSPMSB-2018} for spherical inclusions in three dimensions. The method has spectral convergence for smooth solutions, while the error analysis for close-to-touching inclusions (when the solution is asymptotically singular) still remains open. Rigorous error estimates of the integral equation methods for the asymptotical singular solutions of the perfect conductivity problem, with close-to-touching inclusions of general geometry, still remain open in three dimensions.

The FEM is a widely used numerical technique for solving partial differential equations in domains with complex geometries. However, the existing error analyses of FEM do not address the perfect conductivity problem with close-to-touching inclusions. The FEMs for the conductivity problem with bounded $k$ and large $\vep$ have been well studied in the literature. For example, optimal-order $L^2$ and $H^1$ error estimates for the corresponding elliptic interface problems were established in \cite{babuvska1970finite,chen1998finite,huang2002mortar,sinha2006priori, sinha2009finite,li2010optimal}. 
For perfect conductivity problems (i.e., $k\to\infty$) with possibly close-to-touching inclusions, the first rigorous error estimate of linear FEM uniformly with respect to $\varepsilon$ was established in \cite{LBLH}, where graded meshes and singular basis are designed based on asymptotic estimates of the solution as $\vep \to 0$. These techniques lead to an $H^1$ error bound of $O(h)$ independent of $\vep$. However, the special basis constructed in the paper is limited to low-order convergence, even if faster mesh refinement is used.
Graded meshes are frequently employed in FEM to improve the accuracy of solutions to problems with point singularities \cite{btw22,fm17,Kopteva19}. In contrast, the singularity in the present problem arises from the $l$-th derivative of the solution in the narrow region between inclusions, which scales as $\varepsilon^{-l/2}$ (in two dimensions) and then blows up as $\varepsilon \to 0$. To obtain $\varepsilon$-independent error estimates, the specially designed graded meshes \cite{LBLH} are introduced in this paper, together with  asymptotic estimates of high-order derivatives of solution in the narrow region, providing an effective approach to handle the singular behavior. An alternative way for treating the singularities is the adaptive FEM \cite{cxz09}; however, $\varepsilon$-independent error estimates were not established either.

In this paper, we design high-order convergent methods for the problem through rigorous error analysis, for which a critical component is the derivation of asymptotic estimates for high-order partial derivatives of the solution --- results which are of independent interest and have not been established yet. To fully resolve the singularities in these high-order derivatives, we introduce a set of carefully constructed auxiliary basis functions, which are more intricate than their linear-element counterparts, providing high-order approximations to the solution. The main results of this paper, i.e., the asymptotic regularity estimates of the solution and the design of high-order FEM based on the asymptotic regularity estimates, are presented in the following two sections. 
The proof of the error estimate for the proposed FEM is presented in Section \ref{sec:errorest}. Section \ref{sec:tests} provides numerous numerical examples validating the theoretical results. The proof of the asymptotic estimates of solution is deferred to Section \ref{sec:highorderest}. 
Appendix \ref{app:w_l} contains some lengthy and technical proofs, and an example of explicit construction of auxiliary functions in Appendix \ref{app:aux}, for ease of reference.

\section{Asymptotic regularity estimates}\label{section:regularity}

In this section, we construct auxiliary functions which can capture all the singularities of the solution to \eqref{PDE}. The auxiliary functions are used to establish asymptotic estimates of high-order partial derivatives of the solution, and to design numerical method with high-order convergence and error bound independent of $\varepsilon$. 

For the simplicity of notation, we denote by $C$ a generic positive constant which may be different at each occurrence but is always independent of the distance parameter $\varepsilon$ and the mesh size $h$ of FEM. 
We also denote by $A\lesssim B$ the statement ``$A\le CB$ for some constant $C$'', and denote by $A\sim B$ the statement ``$C^{-1}B\le A\le CB$ for some constant $C$''.

\subsection{Decomposition of the solution}\label{sec:decomp}

We consider the following Dirichlet boundary value problems:
\begin{equation}\label{equ-v1}
\begin{cases}
  \Delta v_{1}=0&\mbox{in}~\Omega,\\
  v_{1}=1&\mbox{on}~ \Gamma_1,\\
  v_{1}=0&\mbox{on}~\Gamma\cup\Gamma_2,
\end{cases}
\qquad\mbox{and}\qquad 
\begin{cases}
      \Delta v_{b}=0&\mbox{in}~\Omega,\\
      v_{b}=c_{2}&\mbox{on}~\Gamma_1\cup\Gamma_2,\\
      v_{b}=\varphi&\mbox{on}~\Gamma.
\end{cases}
\end{equation}
Proceeding to the solution structure, we decompose the solution of \eqref{PDE} as 
\begin{equation}\label{decom}
u(x',x_n)=(c_{1}-c_{2}) v_{1}(x',x_n)+v_{b}(x',x_n) \quad\mbox{for}~x\in\Omega.
\end{equation}
By standard elliptic regularity theory, $\|\nabla^lu\|_{L^\infty(\Omega\setminus\Omega_{1/4})}$ is bounded for all $l\ge1$. Consequently, it suffices to analyze the singular behavior of $\nabla^lu$ in $\Omega_{1/4}$. For the difference $(c_{1}-c_{2})$, it has been proved in \cite[Propositions 2.2, 2.3]{BLY} that 
\begin{equation}\label{c1c2est}
    |c_1-c_2|\lesssim \sqrt{\varepsilon}\quad \text{for}~n=2,\quad |c_1-c_2|\lesssim |\log\varepsilon|^{-1}\quad \text{for}~n=3.
\end{equation}
While, for $v_{b}$, since $v_{b}=c_{2}$ on both $\Gamma_1$ and $\Gamma_2$, indicating no potential difference, \cite[Theorem 1.1]{LLBY} implies that for any $k,t \geq 0$, there exists a constant $\mu \in (0,1)$ such that
\begin{equation}\label{W2infty-vb}
  \big|{\nabla^{k}(v_{b}(x)-c_2)}\big|
  \lesssim \mu^{\frac{1}{\sqrt{\delta(x')}}}
  \lesssim \delta(x')^{t} 
  \lesssim 1 \quad \text{for }x \in \Omega_{1/2}.
\end{equation}

Based upon these estimates, we construct in the next subsection an auxiliary function $\tilde v_l$ in $\Omega_{1/4}$ to capture the singular behavior of the $l$-th order derivatives of $v_1$ defined in \eqref{equ-v1}. This construction will subsequently enable rigorous asymptotic regularity analysis and facilitate the design of high-accuracy numerical methods for the problem.

\subsection{Construction of auxiliary functions (to capture the singularities of $\nabla^{l}v_1$, $l\ge 1$)}\label{sec:aux}

It has been shown in \cite{LBLH} that the function
\begin{equation}\label{def-bar-v1}
	\bar{v}_{1}(x',x_n):=\frac{x_n-\phi_{2}(x')}{\delta(x')} \quad\text{for}~x=(x',x_n)\in\Omega_{1/2} , 
\end{equation}
satisfies that $|\nabla({v}_{1} -\bar{v}_{1})|\lesssim 1$ in $\Omega_{1/4}$, thereby capturing first-order derivative singularities of $v_1$. To extend this idea to high-order singularities, we define $f_1(x',x_n) :=-\Delta\bar{v}_1=-\Delta_{x'}\bar{v}_1$ and iteratively construct: for $k \geq 2$ and $x \in \Omega_{1/2}$,
\begin{align}
{\bar v}_k(x',x_n) &:=-\int_{\phi_{2}(x')}^{\phi_{1}(x')} G(y,x_n) f_{k-1}(x',y) dy, \label{eq:vk}\\
f_k(x',x_n) &:=f_{k-1}(x',x_n)-\Delta\bar{v}_k, \label{eq:fk}
\end{align}
where $G(y,x_n)$ is the Green's function for $-\partial_{x_nx_n}$: 
\begin{align}\label{defgreenfunction3d}
	G(y,x_n)=\frac{1}{\delta(x')}
  \begin{cases}
		(\phi_{1}(x')-x_n)(y-\phi_{2}(x')), \quad \phi_{2}(x')\le y\le x_n,\\
		(\phi_{1}(x')-y)(x_n-\phi_{2}(x')), \quad x_n\le y\le \phi_{1}(x') .
	\end{cases}
\end{align}
The definition \eqref{eq:vk} directly yields boundary properties
\begin{equation}\label{scode}
	{\bar v}_k(x',x_n)=0 \quad \text{on}~ \Gamma^+_{1/2}\cup\Gamma^-_{1/2} 
\end{equation} 
and $\partial_{x_nx_n}{\bar v}_{k}(x',x_n)=f_{k-1}(x',x_n)$, which, combined with \eqref{eq:fk}, implies
\begin{equation}\label{scfest}
	f_k(x',x_n)=-\Delta_{x'}\bar{v}_k \quad \text{in}~\Omega_{1/2}.
\end{equation}
The composite auxiliary function is then defined as
\begin{align}\label{tvl}
\tilde v_l(x',x_n) := \sum_{k=1}^{l} {\bar v}_k(x',x_n) , 
\end{align}
satisfying the following relation
\begin{align}\label{eq:tvl}
  -\Delta \tilde v_l = f_l = -\Delta_{x'}\bar{v}_l(x',x_n) \quad\text{in }\Omega_{1/2}; \quad \tilde v_l = \bar v_1 = v_1 \quad \text{on }\Gamma^\pm_{1/2}.
\end{align}

The following proposition establishes that $\tilde v_l$ captures the $l$-th-order derivative singularities of $v_1$, and its proof is postponed to Section \ref{sec:highorderest} due to its technical nature.

\begin{proposition}\label{thm:v1est}
Let $D$ be a smooth domain and let $\varphi\in C^{l,\gamma}(\Gamma)$ for some integer $l\ge 1$ and $0< \gamma<1$. Then, for convex and smooth inclusions $D_1$ and $D_2$ satisfying the conditions in \eqref{local-geometry}--\eqref{geo2}, the solution of \eqref{equ-v1} satisfies the following estimates: for $x \in \Omega_{1/4}$,
	\begin{alignat}{2}
		&|\nabla^{l}{v}_{1}(x) | + |\nabla^l {\tilde v}_q(x)| \lesssim \delta(x')^{-\frac{l+1}{2}} & \quad &\text{for}~\, q \ge 1, \label{est:v1} \\
    &|\nabla^{s}({v}_{1}-\tilde{v}_l)(x)| \lesssim \delta(x')^{l-s} &&\text{for}~\, 0 \leq s \leq l. \label{est:v1-tv1}
	\end{alignat}
\end{proposition}

\begin{remark}
  For some special domains, the auxiliary functions $\tilde v_l=\sum_{k=1}^l \bar v_k$ can be constructed explicitly. 
  In particular, to facilitate direct reference by readers, an example of explicit auxiliary functions for high-order derivatives is provided in SM2 in the supplementary material.
\end{remark}

\subsection{Regularity estimates (with explicit dependence on $\varepsilon$)}\label{subsec:regularity}
We substitute Proposition \ref{thm:v1est} into decomposition \eqref{decom} and incorporate estimates \eqref{djx}, \eqref{c1c2est}, and \eqref{W2infty-vb} to obtain the following theorem on the regularity of the solution $u$, which is one of the main results of this paper. 

\begin{theorem}\label{thm_highorder}
Under the assumptions of Proposition \ref{thm:v1est}, 
the solution to \eqref{PDE} satisfies, for all $x\in\Omega_{1/4}$ and $l\ge 1$, that
\begin{align}\label{regularity}
    |\nabla^{l} u| \lesssim 
    \left\{
  \begin{aligned}
      \displaystyle
      &\frac{1}{(\varepsilon+|x_1|^2)^{\frac l 2}}&&\mbox{if}\,\,\,n=2,\\[2pt]
      \displaystyle
      &\frac{1}{|\log\varepsilon|(\varepsilon+|x'|^2)^{\frac{l+1}{2}} }
      &&\mbox{if}\,\,\,n=3,
  \end{aligned}\right.
\end{align}
and
\begin{align}\label{regularity2}
\| u - (c_{1}-c_{2}) \tilde v_{l} - c_2 \|_{W^{l,\infty}(\Omega_{1/4})}\lesssim 1. 
\end{align}
\end{theorem}
For $l =1, 2$, the estimates in \eqref{regularity}--\eqref{regularity2} are consistent with those in \cite{Li20,LLY}. Thereby, Theorem \ref{thm_highorder} extends the asymptotic estimates in \cite{Li20,LLY} to arbitrary high-order partial derivatives. Recently, the analogue estimates for the Lam\'{e} system has been established in \cite{DLTZ} with the same pointwise upper bounds as in \eqref{regularity}.

The regularity estimates in \eqref{regularity}, combined with the auxiliary function $\tilde v_l$ and its property  \eqref{regularity2}, will be instrumental in the next section for designing a high-order $\varepsilon$-independent convergent FEM for problem \eqref{PDE}.

\section{Design of high-order FEM}\label{sec:FEM}
In this section, we design graded meshes and auxiliary basis functions, as well as finite element spaces and corresponding FEMs  for the perfect conductivity problem. Extension to the linear elasticity problem is also discussed.  

\subsection{Graded mesh}\label{sec:mesh}

\begin{figure}[htbp]
  \centering
  \begin{tikzpicture}[scale=4.5, semithick]
    \clip (-1,-0.3) rectangle (1,0.3);
    \draw (0,  0.75) circle [x radius=1, y radius=0.7];
    \draw (0, -0.75) circle [x radius=1, y radius=0.7];
    \draw (-0.7,-0.25) -- (-0.7,0.25); 
    \draw (0.7,-0.25) -- (0.7,0.25);
    \draw (-0.8,0) node {\small$\Omega_0$};
    \draw (0.8,0) node {\small$\Omega_0$};
    \draw (-0.4,0.16) node {\small$\Gamma_1$};
    \draw (0.4,-0.16) node {\small$\Gamma_2$};
    \draw (-0.2,-0.06) -- (-0.2,0.06);
    \draw (0.2,-0.06) -- (0.2,0.06); 
    \draw (0,0) node {\small$\Omega_*$};
    \draw (0,-0.25) node {\small$|x'| < 1$};
    \draw [dashed] (-0.7,-0.25) -- (-0.15,-0.25);
    \draw [dashed] (0.7,-0.25) -- (0.15,-0.25);
  \end{tikzpicture}
  \caption{The partition in the region between two inclusions.}
  \vspace{5pt}
\end{figure}
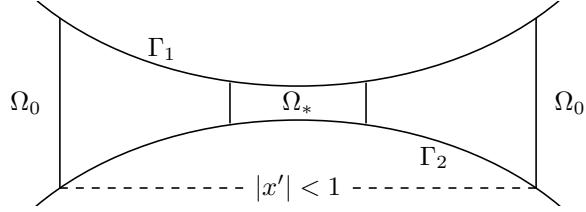

Recalling the definition of $\Omega_r$ in \eqref{def:Or}, and without loss of generality, we let $\Omega_0 := \Omega\setminus\overline{\Omega_1}$ denote the subdomain that is triangulated with uniform mesh size $h>0$. 
To design the graded meshes, we introduce two parameters $\alpha > 0$ and $\kappa \geq 1$, where $\alpha$ represents the rate of mesh refinement, and $\kappa=O(1)$ is a fixed constant which represents the minimum number of triangles/tetrahedra along the $x_n$-axis in the narrow region. 

For given $\varepsilon$ and $h$, we divide the problem into two cases according to whether $\varepsilon \geq (\kappa h)^{\frac{1}{1-\alpha/2}}$ or $\varepsilon < (\kappa h)^{\frac{1}{1-\alpha/2}}$. 

{\bf Case 1:} 
  $\varepsilon \geq (\kappa h)^{\frac{1}{1-\alpha/2}}$. In this case, we define 
  \[
    \Omega_* := \{x \in \Omega_1 : |x'| \leq \varepsilon^{1/2}\}.
  \]
  We triangulate the entire domain $\Omega$ by isoparametric simplicial elements of degree $q$ (cf. \cite{lenoir86}) and with local mesh size 
  \begin{align}\label{eq:hbar1}
    \hbar(x) = 
    \begin{cases}
      h,  & \text{if } x \in \Omega_0, \\
      \varepsilon^{\alpha/2} h,  & \text{if } x \in \Omega_*, \\
      |x'|^\alpha h,      & \text{otherwise.} 
    \end{cases}
  \end{align}
  Denote by $\mcT_h$ the triangulation of $\Omega$ and by $\Omega_h$ the triangulated approximate region, i.e.,
  \[
    \Omega_h = \text{interior of } \bigcup\nolimits_{\hat K\in \mcT_h} \hat K. 
  \]
According to \cite[Section 5.1]{lenoir86}, there exists a one-to-one Lipschitz continuous map $\Phi_h: \overline{\Omega_h} \rightarrow \overline{\Omega}$. We denote by $\Gamma_h = \Phi_h^{-1}(\Gamma)$ and $\Gamma_{j,h} = \Phi_h^{-1}(\Gamma_j)$ the isoparametric approximations to $\Gamma$ and $\Gamma_j=\partial D_j$, respectively. 

{\bf Case 2:} 
$\varepsilon < (\kappa h)^{\frac{1}{1-\alpha/2}}$. In this case, we define 
  \begin{align}\label{def-Omega-s}
    \Omega_* := \{x \in \Omega_1 : |x'| \leq (\kappa h)^{\frac{1}{2-\alpha}}\}
  \end{align}
and triangulate only the subdomain $\Omega_*^c = \Omega\setminus\overline{\Omega_*}$ by isoparametric simplicial elements of degree $q$ and with local mesh size 
  \begin{align}\label{eq:hbar2}
    \hbar(x) = 
    \begin{cases}
      h,      & \text{if } x \in \Omega_0, \\
      |x'|^\alpha h,      & \text{if } x \in \Omega_*^c\setminus \overline{\Omega_0}. 
    \end{cases}
  \end{align}
  Denote by $\mcT_h$ the triangulation of $\Omega_*^c$ and by $\Omega_{*,h}^c$ the triangulated approximate region. Let $\Phi_h$ be the map from $\overline{\Omega_{*,h}^c}$ to $\overline{\Omega_*^c}$ in this case. Let $\Gamma_h = \Phi_h^{-1}(\Gamma)$. The approximated boundaries in this case are denoted by $\Gamma_{j,h} = \Phi_h^{-1}(\Gamma_j \cap \Omega_*^c)$ for $j=1,2$. The numerical solution in the non-triangulated domain $\Omega_*$ will be approximated by a linear combination of some auxiliary functions, see details in the next section.

The approximated boundary of $\Gamma=\partial D$ in both cases is denoted by $\Gamma_h:=\Phi_h^{-1}(\Gamma)$.
From \cite[Proposition 2]{lenoir86}, suppose the boundaries $\Gamma$ and $\Gamma_j,\,j=1,2$ are smooth, there hold
\begin{align}\label{est:iso}
  \|\Phi_h - \mbox{id} \|_{L^\infty} \lesssim h^{l+1}, \;\; 
  \|D\Phi_h - I\|_{L^\infty} \lesssim h^l, \;\;\mbox{and}\;\; 
  \|\det(D\Phi_h) - 1\|_{L^\infty} \lesssim h^l.
\end{align}
The sample graded meshes for two cases, with maximal mesh size $h=0.2$ and refinement parameters $\alpha=1$, $\kappa=2$, are shown in Figure \ref{fig:sample_mesh}. 
\begin{figure}
  \centering
  \includegraphics[width=0.8\textwidth]{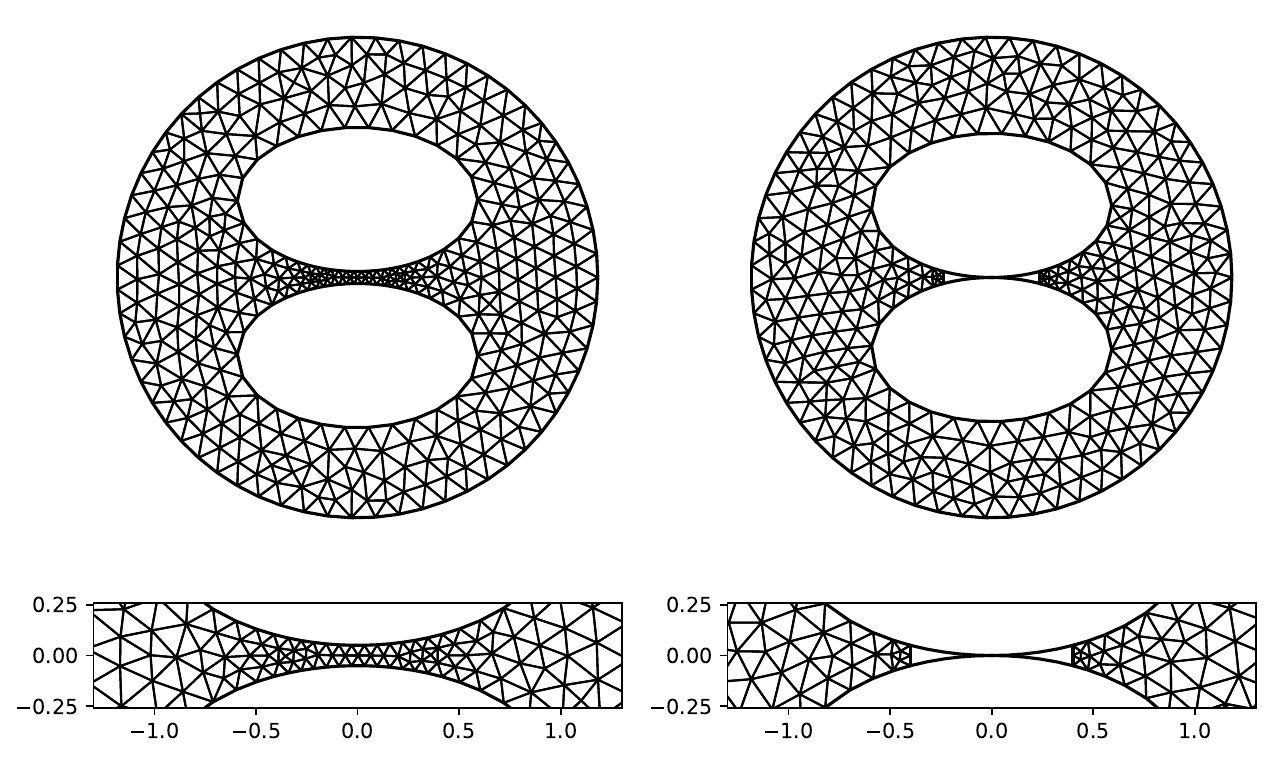}
  \caption{Graded meshes for $\alpha=1$ and $\kappa=2$. Left: Case 1, $\varepsilon=0.1$. Right: Case 2, $\varepsilon=10^{-5}$.}
  \label{fig:sample_mesh}
\end{figure}

\subsection{Finite element spaces} 
In this subsection, we introduce the definitions of finite element spaces for the two different cases. 

In Case 1, we define the finite element space on $\Omega_h$ by
\begin{equation}
  \label{eq:spaceSh1}
  \begin{aligned}
    \hat S_h := \big\{ \hat v \in H^1(\Omega_h) : &\;\hat v|_{\hat K} \in \mcP_p(\hat K) \;\;\mbox{for all } \hat K \in \mcT_h;  \\
    &\;\hat v = \mbox{constant on }\Gamma_{j,h} \mbox{ for }j=1,2
    \big\},
  \end{aligned}
\end{equation}
where $\mcP_p(\hat K)$ denotes the set of all iso-parametric polynomials of degree $\leq p$ on $\hat K$. Namely, if $\hat K$ is a curved triangle/simplex, then $\hat v|_{\hat K}\in \mcP_p(\hat K)$ iff its pull back onto the flat reference triangle/simplex is a polynomial of degree $\leq p$.

In Case 2, we firstly define an interpolation operator on the subdomain ${\Omega_*}$ defined in \eqref{def-Omega-s}, for functions $\hat v$ that are constant on each boundary $\Gamma_{j,h}$ for $j=1,2$. 
The key to designing a high-order numerical scheme for this problem in this case, with high-order convergence and error bound independent of $\varepsilon$, is the construction and use of the auxiliary functions $\tilde{v}_l(x)$ introduced in Section \ref{sec:aux}. 
By utilizing the auxiliary functions defined in \eqref{tvl} with $l=q$, we can define an interpolation operator $\hat I_q^*$ such that the function
\begin{equation}\label{eq:Ih*}
  \hat I_q^* \hat v(x) := \hat v|_{\Gamma_{1,h}} \tilde v_q(x) + \hat v|_{\Gamma_{2,h}} (1-\tilde v_q(x)) = (\hat v|_{\Gamma_{1,h}} - \hat v|_{\Gamma_{2,h}}) \tilde{v}_q(x) + \hat v|_{\Gamma_{2,h}}, \;\; x \in \overline\Omega_*,
\end{equation}
satisfies the following estimate if $\varphi\in C^{l,\gamma}(\Gamma)$ with $l\ge q$ (see \eqref{regularity2} in Theorem \ref{thm_highorder}): 
\begin{align}\label{est:intp:infty}
  \|\nabla^q(u - I_q^* u)\|_{L^\infty(\Omega_*)} 
  \lesssim 1,
\end{align}
where $u$ is the solution to \eqref{PDE} and $I_q^* u := \hat I_q^* (u \circ \Phi_h)$.

We denote $\Gamma_*:= \partial \Omega_* \cap \partial \Omega_*^c$ and define the finite element space on $\Omega_{*,h}^c$ as follows: 
\begin{equation}\label{eq:spaceS*hc}
\begin{aligned}
  \hat S_{*,h}^c := \big\{\hat v \in H^1(\Omega_{*,h}^c) : &\;\hat v|_{\hat K} \in \mcP_p(\hat K) \;\;\mbox{for all } \hat K \in \mcT_h; \\
  &\;\hat v = \mbox{constants on }\Gamma_{j,h} \mbox{ for }j=1,2; \\
  &\;\hat v = I_q^* \hat v \mbox{ at finite element nodes on $\Gamma_*$} 
  \big\}.
\end{aligned}
\end{equation}

The following proposition on the total degrees of freedom was proved in \cite[Theorem 2.6]{LBLH}.
\begin{proposition}
  \label{prop:dof}
  When $\alpha < 1 + \frac{1}{n}$, the total number of degrees of freedom in either finite element space $\hat S_h$ or $\hat S_{*,h}^c$ satisfies $N = O(h^{-n})$.
\end{proposition}

\subsection{Finite element methods}
Denote by $\hat I_h$ the standard Lagrange interpolation operator on the triangulated region $\Omega_h$ or $\Omega_{*,h}^c$. Next, we present the finite element schemes for both cases.

In Case 1, the proposed method for \eqref{PDE} is as follows: Find $\hat u_h \in \hat S_h$ such that $\hat u_h = \hat I_h (\varphi \circ \Phi_h)$ on $\Gamma_h$ and
\begin{align} \label{eq:FEM1}
  \int_{\Omega_h} \hat \nabla \hat u_h \cdot \hat \nabla \hat v_h \rd \hat x = 0 \quad \forall\, \hat v_h \in \big\{\hat v \in \hat S_h : \hat v|_{\Gamma_h} = 0 \big\},
\end{align}
where $\hat\nabla v:= \nabla_{\hat x} v(\hat x)$ and $\hat x = \Phi_h^{-1}(x)$.

In Case 2, the proposed method for \eqref{PDE} is as follows: Find $\hat u_h \in \hat S_{*,h}^c$ such that $\hat u_h = \hat I_h (\varphi \circ \Phi_h)$ on $\Gamma_h$ and
\begin{align}\label{eq:FEM2}
  \int_{\Omega_{*,h}^c} \hat \nabla \hat u_h \cdot \hat \nabla \hat v_h \rd \hat x + \int_{\Omega_*} \nabla \hat I_q^* \hat u_h \cdot \nabla \hat I_q^* \hat v_h \dx = 0 \;\; \forall\, \hat v_h \in \big\{\hat v \in \hat S_{*,h}^c : \hat v|_{\Gamma_h} = 0 \big\},
\end{align}
where the interpolation operator $\hat I_q^*$ is defined in \eqref{eq:Ih*}. 

\subsection{A unified formulation}
The numerical schemes \eqref{eq:FEM1} and \eqref{eq:FEM2} are applicable in practical computations. To facilitate further error analysis, we rewrite them as an equivalent unified formulation on the original domain $\Omega$.

First, we rewrite the finite element spaces defined in \eqref{eq:spaceSh1} and \eqref{eq:spaceS*hc} as follows:
\begin{alignat*}{2}
  S_h &= \{v \in H^1(\Omega) : v\circ\Phi_h \in \hat S_h\} &\quad&\mbox{in Case 1}; \\
  S_{*,h}^c &= \{v \in H^1(\Omega_*^c) : v\circ\Phi_h \in \hat S_{*,h}^c\} &\quad&\mbox{in Case 2}.
\end{alignat*}
Now we define a finite element space on the total domain $\Omega$ in Case 2. 
For functions that are constant on each boundary $\Gamma_j$ for $j=1,2$, we can extend the interpolation operator $\hat I_q^*$ defined in \eqref{eq:Ih*} by setting  
\begin{align*}
  I_q^* v(x) := \hat I_q^*(v \circ \Phi_h) = v|_{\Gamma_1} \tilde v_q(x) + v|_{\Gamma_2}(1-\tilde v_q(x)) \quad \mbox{for }x \in \overline\Omega_*.
\end{align*}
Then we define
\begin{equation}\label{eq:spaceS*h}
  S_{*,h} := \set{ v \in H^1(\Omega_*) : v = \mbox{constants on }\Gamma_j \cap \partial {\Omega}_* \;\;\mbox{and}\;\; v(x) = I_q^* v(x) \;\;\mbox{for } x \in \Omega_* }.
\end{equation}
The finite element space on $\Omega$ in Case 2 is given by
\begin{equation}\label{eq:spaceSh2}
  \begin{aligned}
    S_h := \big\{ v \in L^2(\Omega) &: v = \mbox{constants on }\Gamma_j \mbox{ for }j=1,2; \;\;\mbox{and}\\ 
    &\;\;\; v|_{\Omega_{*}^c} \in S_{*,h}^c, \; v|_{\Omega_*} \in S_{*,h}\big\}.
  \end{aligned}
\end{equation}
We remark that $S_h \subset H^1(\Omega_*)$ and $S_h \subset H^1(\Omega_*^c)$, but $S_h \not \subset H^{1}(\Omega)$. 

Next, we reformulate the numerical schemes \eqref{eq:FEM1} and \eqref{eq:FEM2} in a unified form. Denote the Jacobian matrix by $F(x) = D\Phi_h \circ \Phi_h^{-1}(x)$. In Case 1, we let $u_h = \hat u_h \circ \Phi_h^{-1}$ be the finite element solution pulled back to $\Omega$, and let $A_h(x) = F(x) F(x)^\text{T} \det F(x)$. 
In Case 2, we denote
\begin{align*}
  u_h(x) = 
  \begin{cases}
    \hat u_h \circ \Phi_h^{-1}(x), & x \in \Omega_*^c, \\
    I_q^* u_h(x), & x \in \Omega_*,
  \end{cases} \quad \mbox{and} \quad
  A_h(x) = 
  \begin{cases}
    F(x) F(x)^\text{T} \det F(x), & x \in \Omega_*^c, \\
    I, & x \in \Omega_*.
  \end{cases}
\end{align*} 
By \eqref{eq:spaceSh2}, we have $u_h \in S_h$. Moreover, we define the following subsets of $S_h$: 
\begin{align}
  S_h^\varphi = \big\{v \in S_h : v|_{\partial D} = \hat I_h (\varphi \circ \Phi_h) \circ \Phi_h^{-1}\big\} \quad\mbox{and}\quad S_h^0 = \big\{v \in S_h : v|_{\partial D} = 0\big\}. \label{def:S_h}
\end{align}
Then the proposed methods in \eqref{eq:FEM1} and \eqref{eq:FEM2} can be rewritten as the following unified form:
\begin{align}
  \mbox{Find } u_h \in S_h^\varphi \quad \mbox{s.t.}\quad \int_{\Omega_*^c \cup \Omega_*} (A_h \nabla u_h \cdot \nabla v_h) \rd x = 0 \quad \forall v_h \in S_h^0.\label{FEM}
\end{align}

\begin{remark}\label{rmk:jump}
In Case 2 and two dimensions, since $\Gamma_*$ consists of two line segments and the triangulation of $\Omega_*^c$ is fully fitted with $\Gamma_*$, it follows that $v_h$ is continuous on $\Gamma_*$ and therefore $v_h \in H^1(\Omega)$. 
However, in Case 2 and three dimensions, since $\Gamma_*$ is the lateral surface of a cylinder and the triangulation of $\Omega_*^c$ is unfitted with $\Gamma_*$, it follows that the jump $\jump{v_h}:= v_h|_{\Omega_*^c} - v_h|_{\Omega_*}$ of $v_h$ on $\Gamma_*$ does not vanish and therefore $v_h$ is generally not in $H^1(\Omega)$. 
\end{remark}

Finally, for any $v_h \in S_h^0$, multiplying \eqref{PDE} by $v_h$ and integration by parts gives
\begin{align}
  \int_{\Omega_*^c \cup \Omega_*} \nabla u \cdot \nabla v_h \rd x 
  = \int_{\Gamma_*} \partial_\nu u \cdot \jump{v_h} \rd s, \label{eq:vp}
\end{align}
where $\nu$ is the unit normal pointing to $\Omega_*$.
Then, the difference between \eqref{eq:vp} and \eqref{FEM} gives the following error equation: 
\begin{equation}\label{eq:ortho}
  \int_{\Omega_*^c \cup \Omega_*} A_h \nabla (u_h-u) \cdot \nabla v_h \rd x = \int_{\Omega_*^c \cup \Omega_*} (I - A_h) \nabla u \cdot \nabla v_h \rd x - \int_{\Gamma_*} \partial_\nu u \cdot \jump{v_h} \rd s.
\end{equation}

\subsection{FEM for linear elasticity problem} 


In addition to the conductivity problem, we will also consider the linear elasticity problem. In this case, the displacement field 
$$u=(u_1,u_2,\cdots,u_n)^\text{T}: D \to \mathbb{R}^n$$ 
is described by the Dirichlet problem for Lam{\'e} system 
\begin{equation}\label{eq:model0}
  \begin{cases}
    \nabla \cdot \left((\chi_\Omega \msC^0 + \chi_{D_1\cup D_2} \msC^1)e(u) \right) = 0 & \text{in }D, \\
    u = \varphi & \text{on }\Gamma,
  \end{cases}
\end{equation}
where $\chi_G$ is the characteristic function on domain $G \subset D$,
$e(u) = \frac{1}{2} (\nabla u + (\nabla u)^\text{T})$ is the strain tensor, 
and the elasticity tensors in background and inclusions are given by
\[
  \mathsf{C}^0_{ijkl} = \lambda \delta_{ij} \delta_{kl} + \mu(\delta_{ik}\delta_{jl} + \delta_{il} \delta_{jk})
  \quad \text{and} \quad
  \mathsf{C}^1_{ijkl} = \lambda_1 \delta_{ij} \delta_{kl} + \mu_1(\delta_{ik}\delta_{jl} + \delta_{il} \delta_{jk}),
\]
where $i,j,k,l = 1, 2, \cdots, n$ and $\delta_{ij}$ is the Kronecker symbol.

In high-contrast composite media, the strong concentration of $\nabla u$ typically occurs when $\vep$ is sufficiently small. 
To quantitatively analyze the influence of $\vep$ on this concentration phenomenon, we assume the Lam{\'e} constants in $D_1 \cup D_2$ degenerate to infinity and consider this extreme case. 
Let $u \in H^1(D)$ be the limiting solution of \eqref{eq:model0} as $\min\{\mu_1, n\lambda_1 + 2\mu_1\} \to \infty$, 
which yields 
\begin{equation}\label{eq:euD1D2}
  e(u) = 0 \quad \text{in } D_1 \cup D_2.
\end{equation}
We introduce a linear space of rigid displacement in $\bbR^n$, $\Psi = \left\{\psi \in C^1(\bbR^n) : e(\psi) = 0 \right\}$ with a basis $\{\psi_l : l = 1, 2, \cdots, {n(n+1)}/{2}\}$, which is 
\begin{align*}
  \Psi = \Span \left\{ e_i, \;x_j e_k - x_k e_j : 1 \leq i \leq n, \;1 \leq j < k \leq n \right\}, 
\end{align*}
where $\{e_i : 1 \leq i \leq n\}$ denotes the standard basis in $\bbR^n$. 
Then \eqref{eq:euD1D2} implies that 
\[
  u = \sum\nolimits_{l=1}^{n(n+1)/2} u_{l,j}\psi_l \quad \text{in }D_j \;\; \text{for }j=1,2. 
\]
By continuity conditions on $\Gamma_j$, the Lam{\'e} system \eqref{eq:model0} becomes the problem in $\Omega$:
\begin{equation}\label{eq:model1}
  \begin{cases}
    \mathcal{L}_{\lambda,\mu} u := \nabla \cdot (\msC^0 e(u)) = 0 & \text{in }\Omega, \\
    u = \sum_{l=1}^{n(n+1)/2} u_{l,j}\psi_l & \text{on }\Gamma_j \text{ for } j=1,2, \\
    \int_{\Gamma_j} \frac{\partial u}{\partial \nu} \cdot \psi_l = 0 & \text{for }l = 1, \cdots, n(n+1)/2 \text{ and }j = 1,2, \\
    u = \varphi & \text{on }\partial D,
  \end{cases}
\end{equation}
where $u_{l,j}$ are free constants to be determined by the third equation above and 
\[
  \frac{\partial u}{\partial \nu} := (\msC^0 e(u)) \vec{n} = \lambda (\nabla \cdot u) \vec{n} + 2\mu e(u) \vec{n} \quad \text{on }\Gamma_j,
\]
where $\vec{n}$ denotes the unit outward normal of $D_j,\, j=1, 2$.
An analogue of Theorem \ref{thm_highorder} for the linear elasticity problem is presented in \cite[Theorems 1.1 and 1.3]{DLTZ}. 

The finite element spaces for linear elasticity problem are similar to those for perfect conductivity problem (see \eqref{eq:spaceSh1} and \eqref{eq:spaceSh2}), except that each instance of the condition ``$v = $ constants on $\Gamma_j$'' is replaced by ``$v = \sum_{l=1}^{n(n+1)/2} v_{l,j} \psi_l$ on $\Gamma_j$, where $v_{l,j}$ are free constants of $v$''. In this case, the interpolation in the narrow region is defined by
\begin{align}
  I_q^* v := \sum_{j=1}^2 \sum_{l=1}^{n(n+1)/2} v_{l,j} \Lambda_{l,j}^{(q)}(x) \;\; \text{with}\;\; \Lambda_{l,j}^{(q)}(x) 
  := \psi_l(x) \sum_{m=1}^q \mathbf{v}_{l,j}^{(m)}(x), \;\; x \in \overline{\Omega}_*, \label{eq:Iq-elas}
\end{align} 
where the functions $\mathbf{v}_{l,j}^{(m)}(x)$ are derived iteratively and defined in detail in \cite{DLTZ}, in which $\mathbf{v}_{l,1}^{(m)}$ is denoted by $\mathbf{v}_{l}^{m}$, and $\mathbf{v}_{l,2}^{(m)}(x):=\psi_l(x)-\mathbf{v}_{l}^{m}(x)$. 


We next uniformly define the Jacobian determinant by $J(x) = \det F(x)$ in Case 1 and by
\begin{align*}
  J(x) = 
  \begin{cases}
    \det F(x), & x \in \Omega_*^c, \\
    1, & x \in \Omega_*,
  \end{cases} \quad \mbox{in Case 2. }
\end{align*}
Moreover, for matrix $\Theta$, we define $A_h \Theta = \frac{1}{2} J(x)^{-1} (F F^\text{T} \Theta + F \Theta^\text{T} F)$ in Case 1 and
\begin{align*}
  A_h \Theta = 
  \begin{cases}
    \frac{1}{2} J(x)^{-1} (F F^\text{T} \Theta + F \Theta^\text{T} F), & x\in \Omega_*^c, \\
    \frac{1}{2} (\Theta + \Theta^\text{T}), & x \in \Omega_*,
  \end{cases} \quad \mbox{in Case 2. }
\end{align*}
Then FEM for the linear elasticity problem \eqref{eq:model1} is written as a unified form: 
\begin{align}
  &\mbox{Find $u_h \in S_h^\varphi$ \;\;s.t. } \notag\\
  \int_{\Omega_*^c \cup \Omega_*} &\big( \lambda J(x) (\nabla\cdot u_h)(\nabla\cdot v_h) + 2\mu (A_h \nabla u_h, \nabla v_h) \big) \dx = 0 \;\; \forall v_h \in S_h^0, \label{FEM:elas}
\end{align}
where $u_h(x) = I_q^* u_h(x)$ in $\Omega_*$.

\subsection{Main result}
At the end of this section, we present the main theoretical result of this paper, which concerns error estimates for finite element schemes \eqref{FEM} and \eqref{FEM:elas}. 

\begin{theorem}\label{thm:error}
Under the conditions of Proposition \ref{thm:v1est}, we further assume that the finite elements of degree $p$ and the order $q$ defined in the interpolation \eqref{eq:Ih*} satisfy that $p\leq l-1$ and $[2p(1-\frac{1}{n})+3-n]/4<q\le l$. Moreover, let $\alpha \in (\alpha_{\rm min}, 1+\frac{1}{n})$ with 
\begin{align}\label{alphamin}
\alpha_{\rm min} = \max\set{1+ \tfrac{n-3}{2p}, 2- \tfrac{4q+n-3}{2p}}.
\end{align}
Then the finite element solution $u_h$ given by \eqref{FEM} or \eqref{FEM:elas} has the following error bound in approximating the solution of \eqref{PDE} or \eqref{eq:model1}{\rm:}
  \begin{align}\label{est:error}
    \|u-u_h\|_{L^2(\Omega)} + \|\nabla(u-u_h)\|_{L^2(\Omega_* \cup \Omega_*^c)} \leq C h^p,
  \end{align}
  where the constant $C > 0$ is independent of $\vep$. 
\end{theorem}

\begin{remark}\label{rmk:pql}
	Let us take a careful look at the relations between the constants $l$, $p$, $q$, and $\alpha$. 

	First, recalling that $\varphi \in C^{l,\gamma}(\partial D)$ with $l\ge 2$ and boundaries $\Gamma$, $\Gamma_j,\,j=1,2,$ are sufficiently smooth (e.g., $C^l$-smooth), then the solution to \eqref{PDE} satisfies the regularity $u\in H^1(\Omega) \cap C^{l}(\overline{\Omega})$ and asymptotic estimate presented in Theorem \ref{thm_highorder}. We also note that it suffices to consider the problems with $\varphi 
  \in C^0(\partial D)$ by using the maximum modulus estimate and interior estimate.

	Second, the order of polynomials in finite element space $S_h$ is chosen as $p \leq l-1$ due to the regularity of solutions. The condition $q\le l$ is required by \eqref{est:intp:infty}. The lower bound of the order $q$ of interpolation operator in \eqref{eq:Ih*} is determined by $p$, whereas the mesh parameter $\alpha$ depends on $q$ and $p$. 
  Indeed, the mesh condition $\alpha_{\rm min} < \alpha < 1+ \frac{1}{n}$ should be satisfied to guarantee the optimality of convergence rate with respect to both meshsize and degrees of freedom, i.e., the finite element solution converges with an $H^1$-error bound of $O(h^p) \eqsim O(N^{-p/n})$.
  This mesh condition yields that $q$ should be selected such that $\alpha_{\rm min} < 1+\frac{1}{n}$, leading to the lower bound $[2p(1-\frac{1}{n})+3-n]/4$ on $q$, which is, $q > (p+1)/4$ when $n=2$ and $q > p/3$ when $n=3$.
  
  Table {\ref{tab:pqal}} provides the parameter selection criteria of $q$ and $\alpha$ for $p \leq 5$.
  \begin{table}[htbp]
    \centering
    \caption{Some choices of $q$ and the ranges of $\alpha$ for $p \leq 5$.}
    \label{tab:pqal}
    \begin{tabular}{|c|c|cc|}
      \hline
      \multirow{2}{*}{\begin{minipage}{1cm}\centering$p$\end{minipage}} & \multirow{2}{*}{\begin{minipage}{1cm}\centering$q$\end{minipage}} & \multicolumn{2}{c|}{$\alpha$}                                            \\ \cline{3-4} 
                          &                      & \multicolumn{1}{c|}{\begin{minipage}{4cm}\centering$n=2$\end{minipage}}                     & \begin{minipage}{4cm}\centering$n=3$\end{minipage}                   \\ \hline
      $1$ & \multirow{2}{*}{$1$} & \multicolumn{1}{c|}{$1/2 < \alpha < 3/2$} & \multirow{4}{*}{$1 < \alpha < 4/3$} \\ \cline{1-1} \cline{3-3}
      $2$                  &                      & \multicolumn{1}{c|}{$5/4 < \alpha < 3/2$}   &                         \\ \cline{1-3}
      $3$                  & \multirow{4}{*}{$2$} & \multicolumn{1}{c|}{$5/6 < \alpha < 3/2$}   &                         \\ \cline{1-1} \cline{3-3}
      $4$                  &                      & \multicolumn{1}{c|}{$9/8 < \alpha < 3/2$}   &                         \\ \cline{1-1} \cline{3-4} 
      $5$                  &                      & \multicolumn{1}{c|}{$13/10 < \alpha < 3/2$} & $6/5 < \alpha < 4/3$ \\ 
      \hline
    \end{tabular}%
  \end{table}
\end{remark}

\section{Proof of Theorem \ref{thm:error}}\label{sec:errorest}

In this section, we first define an interpolation operator $I_h : H_c^1(\Omega) \cap C(\overline\Omega) \to S_h$
and establish an interpolation error estimate for $u-I_h u$, where $u$ denotes the solution to problem \eqref{PDE}. We then provide a rigorous proof of the main result stated in Theorem~\ref{thm:error}.

\subsection{Interpolation operator}\label{sec:Ih}
In Case 1, the interpolation operator is defined in the standard way, namely, $I_h v := \hat I_h(v \circ \Phi_h) \circ \Phi_h^{-1}$, and obviously, $I_h v \in S_h$ for all $v \in H_c^1(\Omega) \cap C(\overline\Omega)$. 
In Case 2, we first modify the standard operator $\hat I_h$ to an interpolation operator on $\overline{\Omega_{*,h}^c}$ satisfying
\begin{alignat*}{2}
  {\rm (i)}\;   & \hat I_h v|_{\hat K} \in \mcP_p(\hat K) & \quad &\mbox{for all } \hat K \in\mcT_h; \\
  {\rm (ii)}\;  & \hat I_h v(x) = v(x) & &\mbox{if $x \in \overline{\Omega_{*,h}^c} \setminus \Gamma_*$ is a finite element node of $\mcT_h$}; \\
  {\rm (iii)}\; & \hat I_h v(x) = I_q^* v(x) & &\mbox{if $x \in \Gamma_*$ is a finite element node of $\mcT_h$},
\end{alignat*} 
where $I_q^*$ is defined in \eqref{eq:Ih*}.
Then, the interpolation operator 
in this case is defined by
\begin{equation}\label{eq:intp2}
  I_h v(x) :=
  \begin{cases}
    \hat I_h(v \circ \Phi_h) \circ \Phi_h^{-1}(x) & \mbox{for } x \in \overline{\Omega_{*}^c}, \\
    I_q^* v(x) & \mbox{for } x \in \Omega_*.
  \end{cases}
\end{equation}
Obviously, for any $v \in H_c^1(\Omega) \cap C(\overline\Omega)$, we have $I_h v |_{\Omega_{*}^c} \in S_{*,h}^c$ and $I_h v |_{\Omega_*} \in S_{*,h}$, then $I_h v \in S_h$. 

\subsection{The interpolation estimate}

\begin{lemma}\label{lem:intp}
Under the conditions of Theorem \ref{thm:error}, the solution $u$ to problem \eqref{PDE} satisfies the following estimate: 
	\begin{align}
		\|\nabla (u - I_h u)\|_{L^2(\Omega_*^c\cup\Omega_*)} \lesssim h^p. \label{est:intp}
	\end{align}
\end{lemma}
\begin{proof}
We recall the definitions of $\Omega_0 := \Omega\setminus\overline{\Omega_1}$ and $\Omega_*$ in Section \ref{sec:mesh}, and rewrite the definition of the latter equivalently as 
\begin{align*}
\Omega_* = \Big\{x \in \Omega_1 : |x'| \leq R_* := \max\big\{\vep^{1/2}, (\kappa h)^{\frac{1}{2-\alpha}}\big\}\Big\}.
\end{align*}
For any domain $G \subset \mathbb{R}^n$, we denote $G' := \mbox{interior of } \bigcup_{K\in \mcT_h, K\cap \overline{G} \not= \emptyset} K$. Thus $\Omega'_*$ is slightly larger than $\Omega_*$. 
  
	{\it Step 1.} When $K \in \mcT_h$ and $K \cap \Omega_0 \neq \emptyset$, by the standard interpolation estimates, 
  \begin{equation}\label{est:intp0}
    \begin{aligned}
      \|\nabla (u - I_h u)\|_{L^2(\Omega'_0)} & \lesssim \sum_{K \cap \Omega_0 \neq \emptyset} \|\nabla (u - I_h u)\|_{L^2(K)} \lesssim h^{p} \| u \|_{H^{p+1}(\Omega'_0)} \lesssim h^p.
    \end{aligned}
  \end{equation}

  {\it Step 2.} When $K \in \mcT_h$, $K \cap \Omega_1 \neq \emptyset$, and $K \cap \overline\Omega_* = \emptyset$, we use Theorem \ref{thm_highorder} and $\hbar(x) = |x'|^\alpha h$ in $\Omega_{1}\backslash\Omega_*$, which imply the following estimates: 
  \begin{align*}
    \| \nabla (u - I_h u) \|_{L^2(\Omega'_1\setminus\Omega'_*)}^2 & \leq \sum_{K \cap \Omega_1 \neq \emptyset,\,K \cap \overline\Omega_* = \emptyset} \int_K |\nabla (u - I_h u)|^2 \rd x \\
    & \lesssim \sum_{K \cap \Omega_1 \neq \emptyset,\,K \cap \overline\Omega_* = \emptyset} \int_K (|x'|^\alpha h)^{2p} \| u \|_{W^{p+1}_{\infty}(K)}^2 \rd x \\
    \mbox{(Theorem \ref{thm_highorder} is used)}\; & \lesssim h^{2p} \sum_{K \cap \Omega_1 \neq \emptyset,\,K \cap \overline\Omega_* = \emptyset} \int_K |x'|^{2 \alpha p} (\vep + |x'|^2)^{-(p+n-1)} \rd x \\
    \mbox{(since $|x'|\ge R_* \geq \varepsilon^{1/2}$)}\; & \lesssim h^{2p} \int_{R_* \le |x'| \lesssim 1} \rd x' \int_{\phi_2(x')}^{\phi_1(x')} |x'|^{2(\alpha p - p - n + 1)} \rd x_n \\
    \mbox{(since $\phi_1-\phi_2 \eqsim \varepsilon + |x'|^2$)}\; & \lesssim h^{2p} \int_{R_* \le |x'| \lesssim 1} |x'|^{2(\alpha p - p - n + 2)} \rd x' \\
    & \lesssim h^{2p} \left(1 + R_*^{2(\alpha-1)p - n + 3}\right).
  \end{align*}
  Since {$\alpha > 1 + {(n-3)}/{2p}$}, we have $2(\alpha-1)p-n+3 > 0$ and therefore 
  \begin{equation}\label{est:intp1}
    \begin{aligned}
      \| \nabla (u - I_h u) \|_{L^2(\Omega'_1\setminus\Omega'_*)} \lesssim h^p.
    \end{aligned}
  \end{equation}

  {\it Step 3.} We estimate the interpolation errors in the narrow region $\Omega'_*$ for two cases separately. 
	In Case 1 where $\vep \geq (\kappa h)^{\frac{1}{1-\alpha/2}}$ (and $\kappa$ is a constant), the mesh size in $\Omega_*'$ is $\vep^{\alpha/2}h$. Noting $|x'| \lesssim \vep^{1/2}+ \vep^{\alpha/2}h$ for $x \in \Omega_*'$ and $\vep^{\alpha/2}h \le \kappa^{-1} \vep \lesssim \vep^{1/2}$, we have
  \begin{equation}\label{est:intp*1}
    \begin{aligned}
      \| \nabla (u - I_h u) \|_{L^2(\Omega'_*)}^2 & \lesssim (\vep^{\alpha/2}h)^{2p} \| u \|_{H^{p+1}(\Omega'_*)}^2 \\
      \mbox{(Theorem \ref{thm_highorder} is used)}\; & \lesssim h^{2p} \vep^{\alpha p} \int_{|x'| \lesssim \vep^{1/2} + \vep^{\alpha/2}h} \rd x' \int_{\phi_2(x')}^{\phi_1(x')} \vep^{-(p+n-1)} \rd x_n \\
      \mbox{(since $\phi_1-\phi_2 \eqsim \varepsilon + |x'|^2 \lesssim \varepsilon$)}\; & \lesssim h^{2p} \vep^{(\alpha-1) p + 2 - n} \int_{|x'|\lesssim \vep^{1/2}} \rd x' \\
      &\lesssim h^{2p} \vep^{(\alpha-1) p - (n-3)/2} \\  
      &\lesssim h^{2p},
    \end{aligned}
  \end{equation}
  where we have used $\alpha > 1 + {(n-3)}/{2p}$ again in the last inequality. 
	

  In Case 2 where $\vep < (\kappa h)^{\frac{1}{1-\alpha/2}}$, for the elements satisfying $K \in \mcT_h$ and $K \cap \overline\Omega_* \neq \emptyset$, we have
  \begin{align*}
    \sum_{K \cap\overline\Omega_* \neq \emptyset} &\|\nabla (u-I_h u)\|_{L^2(K)}^2 \\
    &\lesssim \sum_{K \cap\overline\Omega_* \neq \emptyset}\|\nabla (u- I_K u)\|_{L^2(K)}^2 + \sum_{K \cap\overline\Omega_* \neq \emptyset}\|\nabla (I_K u - I_h u)\|_{L^2(K)}^2, 
  \end{align*}
  where $I_K$ denotes the Lagrange interpolation operator on $K$. 
  Similar to the proof of \eqref{est:intp1}, the first term in right-hand-side of the above inequality is bounded by $O(h^{2p})$. For the second term, we note $I_K u - I_h u$ is a finite element function on $K$, whose degrees of freedom are equal to the interpolation $u-I_q^* u$ at the finite element nodes on $\Gamma_*$, and vanish at all other nodes. 
  Moreover, the mesh size $\hbar(x')$ near $\Gamma_*$ satisfies $|x'|^\alpha h \sim (\kappa h)^{\frac{\alpha}{2-\alpha}}h \sim h^{\frac{2}{2-\alpha}}$, and therefore,
  \begin{align*}
    \sum_{K \cap\overline\Omega_* \neq \emptyset} &\|\nabla (I_K u - I_h u)\|_{L^2(K)}^2 \\
    &\lesssim \Big( h^{-\frac{2}{2-\alpha}} \|u - I_q^* u\|_{L^\infty(\Gamma_*)}\Big)^2
    \int_{|x'|\sim h^{\frac{2}{2-\alpha}}} \diff x'\int_{\phi_2(x')}^{\phi_1(x')} \diff x_n.
  \end{align*}
  Noting $h^{\frac{2}{2-\alpha}} \lesssim h^{\frac{1}{2-\alpha}}$, we get 
  \begin{equation}\label{est:intp*2}
    \begin{aligned}
      &\;\| \nabla (u - I_h u) \|_{L^2(\Omega'_*)}^2  \\
      \lesssim &\;\sum_{K \cap\overline\Omega_* \neq \emptyset} \| \nabla (u - I_h u) \|_{L^2(K)}^2 + \| \nabla (u - I_h u) \|_{L^2(\Omega_*)}^2 \\
      \lesssim &\;h^{2p} + \Big( h^{-\frac{4}{2-\alpha}} \|u - I_q^* u\|_{L^\infty(\Gamma_*)}^2 + \| \nabla (u - I_q^* u) \|_{L^\infty(\Omega'_*)}^2  \Big) \int_{|x'| \lesssim h^{\frac{1}{2-\alpha}}} \rd x' \int_{\phi_2}^{\phi_1} \rd x_n. 
    \end{aligned}
  \end{equation}
	A simple calculation yields
	\begin{align}
		\int_{|x'| \lesssim  h^{\frac{1}{2-\alpha}}} \rd x' \int_{\phi_2(x')}^{\phi_1(x')} \rd x_n 
		\lesssim 
		\int_{|x'| \lesssim  h^{\frac{1}{2-\alpha}}} (\vep + |x'|^2) \rd x'
		\lesssim 
		h^{\frac{n+1}{2-\alpha}}. \label{est:vol}
	\end{align}
  where the last inequality follows from $\vep < (\kappa h)^{\frac{1}{1-\alpha/2}}$ in Case 2 and $|x'|\lesssim h^{\frac{1}{2-\alpha}}$ in the integral. It remains to estimate the factor $\| {\nabla^s} (u - I_q^* u) \|_{L^\infty(\Omega'_*)}$ in \eqref{est:intp*2} {with $s=0,1$.} 

  By the decomposition \eqref{decom} and the definition \eqref{eq:Ih*}, it holds
  \begin{align*}
    u - I_q^* u = (c_1-c_2)(v_1 - \tilde v_q) + (v_b-c_2),
  \end{align*}
  where $c_j,\,j=1,2,$ are free constants of $u$ on $\Gamma_j$. Using the second result of Proposition \ref{thm:v1est} and \eqref{c1c2est}--\eqref{W2infty-vb}, for $1 \le q \le l$, $s=0,1$, and $x \in \Omega_{1/4}$, we have 
  \begin{align}
    |{\nabla^s} (u - I_q^* u)(x) | &\leq  |c_1-c_2| |{\nabla^s} (v_1 - \tilde v_q)(x)| + |{\nabla^s} (v_b-c_2)| 
    \lesssim \delta(x')^{{q-s}}. \label{est:res}
  \end{align}
  Therefore, from \eqref{est:intp*2}, \eqref{est:vol}, and \eqref{est:res}, 
  we arrive at
	\begin{align}
		\| \nabla (u - I_h u) \|_{L^2(\Omega'_*)}^2 &\lesssim  h^{2p} + h^{\frac{n+1}{2-\alpha}} \Big( {h^{-\frac{4}{2-\alpha}}\big(\vep + (\kappa h)^{\frac{2}{2-\alpha}}\big)^{{2q}}} + \big(\vep + (\kappa h)^{\frac{2}{2-\alpha}}\big)^{{2(q-1)}} \Big) \notag \\
    &\lesssim h^{2p} + (\kappa h)^{\frac{{4q+n-3}}{2-\alpha}} \label{est:intp*3} \\
    &\lesssim h^{2p}, \notag
	\end{align}
  where we have used $\alpha > \alpha_{\rm min} \geq 2-\frac{{4q+n-3}}{2p}$ 
  in the last inequality.
  
  Combining \eqref{est:intp0}, \eqref{est:intp1}, \eqref{est:intp*1}, and 
  \eqref{est:intp*3}, we conclude the proof of this lemma.
\end{proof}
\begin{remark}
  The interpolation error estimate \eqref{est:intp} remains valid for linear elasticity problem. Since the proof follows along similar lines, we omit the details here.
\end{remark}

\subsection{Finite element error estimate}
Next, we focus on the perfect conductivity problem. We first present an estimate for free constants which is useful in the error estimates, and then give a proof of Theorem \ref{thm:error} by using the interpolation error estimate in Lemma \ref{lem:intp}. 

\begin{lemma}\label{lem:fc}
For $v \in H_c^1(\Omega) = \left\{ v \in H^{1}(\Omega) : v \text{ is constant on } \Gamma_j \text{ for } j=1,2 \right\}$, the following estimates hold:
\begin{align}
    \abs{v|_{\Gamma_j}} &\lesssim \|\nabla v\|_{L^1(\Omega_0)} + \|v\|_{L^1(\Gamma)} . \label{est:gammaj}
\end{align}
\end{lemma}

\begin{proof}
  Let $\check\Gamma_j = \Gamma_j \cap \Omega_0$ with $\Omega_0=\Omega\backslash\overline\Omega_1$. 
  Since $v|_{\check\Gamma_j}$ is constant, it follows that 
  \[
  \bv{v|_{\check\Gamma_j}} \lesssim \|v\|_{L^1(\check\Gamma_j)} 
  \lesssim \|v\|_{W^{1,1}(\Omega_0)} ,  
  \]
  where the last inequality is the trace inequality. Then 
  \eqref{est:gammaj} follows from 
  \begin{align}
    \|v\|_{W^{1,1}(\Omega_0)}
  \le \|\nabla v\|_{L^1(\Omega_0)} + \|v\|_{L^1(\Gamma)},  \label{Friedrichs}
  \end{align}
  which is the Friedrichs-type inequality, see, e.g., \cite[eq. (6.11.2) in Section 6.11]{Mazya11}.
\end{proof}

Next, we present the error estimate for the numerical method in \eqref{FEM}.

\begin{proof}[Proof of Theorem \ref{thm:error} for the perfect conductivity problem]
  Let $\rho = u - I_h u$ and $\theta_h = u_h - I_h u$. 
  For any $v_h \in S_h^0$, by error equation \eqref{eq:ortho}, we start with
  \begin{equation}
    \begin{aligned}
      &\;\quad\int_{\Omega_*^c \cup \Omega_*} A_h \nabla\theta_h \cdot \nabla v_h \rd x \\
      &= \int_{\Omega_*^c \cup \Omega_*} (I-A_h)\nabla u \cdot\nabla v_h \rd x + \int_{\Omega_*^c \cup \Omega_*} A_h \nabla\rho \cdot \nabla v_h \rd x - \int_{\Gamma_*} \partial_\nu u \cdot \jump{v_h} \rd s. 
    \end{aligned}\label{est1}
  \end{equation}
  For simplicity, we denote $\|\cdot\|_{L^t} = \|\cdot\|_{L^t(\Omega_*^c \cup \Omega_*)}$ for $t \geq 1$. 
  Using \eqref{est:iso} and noting that 
  \begin{align*}
    A_h - I = (F-I)(F^{\rm T} \det F) + (F^{\rm T} - I)\det F + (\det F - 1)I,
  \end{align*}
  it is easy to verify that $\|I - A_h\|_{L^\infty(\Omega_*^c)} \lesssim h^p$. Let $v_h = \theta_h$ in \eqref{est1}, we have
  \begin{equation}
    \begin{aligned}
      \|\nabla \theta_h\|_{L^2}^2 &\lesssim h^p \|\nabla u\|_{L^2} \|\nabla \theta_h\|_{L^2} + \|\nabla \rho\|_{L^2} \|\nabla \theta_h\|_{L^2} + \|\nabla u\|_{L^2(\Gamma_*)} \|[{\theta_h}]\|_{L^2(\Gamma_*)}. 
    \end{aligned}\label{est2}
  \end{equation}
  Since $\|\nabla \rho\|_{L^2} \lesssim h^p$ (as a result of Lemma \ref{lem:intp}) and $\|\nabla u\|_{L^2} \lesssim 1$, it follows that  
  \begin{align}
    \|\nabla \theta_h\|_{L^2}^2 \lesssim h^{2p} + \|\nabla u\|_{L^2(\Gamma_*)}\|[{\theta_h}]\|_{L^2(\Gamma_*)}. \label{est3.0}
  \end{align}
  It remains to estimate the last term in \eqref{est3.0}. 
  According to Remark \ref{rmk:jump}, we only need to consider Case 2 in three dimensions, that is, when $\vep < (\kappa h)^{\frac{2}{2-\alpha}}$ and $n=3$. In all other cases, we have $[{\theta_h}]|_{\Gamma_*} = 0$, and hence, this term vanishes. 
  Next, we denote $\theta_h^+ := \theta_h|_{\Omega_*^c}$ and $\theta_h^- := \theta_h|_{\Omega_*}$. Let $\theta_h^* := I_q^* \theta_h = \hat I_q^*(\theta_h \circ \Phi_h)$ be the vertical interpolation defined in \eqref{eq:Ih*}. 
  Recalling the definition of interpolation operator $I_h$ defined in Section \ref{sec:Ih}, it is clear that $\theta_h^- = I_q^* u_h - I_q^* u = \theta_h^*$ on $\Gamma_*$. 
  Furthermore, noting $\theta_h^+$ is a finite element function in $\Omega_*^c$ and $\theta_h^+ = \theta_h^*$ at all finite element nodes on $\Gamma_*$, we have $\theta_h^+|_{\Gamma_*} = \hat I_h (\theta_h^* \circ \Phi_h) \circ \Phi_h^{-1} = I_h \theta_h^*$.
  Therefore, $\jump{\theta_h}=\theta_h^* - I_h \theta_h^*$ on the interface $\Gamma_*$, and \eqref{est3.0} becomes
  \begin{equation}
    \begin{aligned}
      \|\nabla \theta_h\|_{L^2}^2 
      &\lesssim h^{2p}  + \|\nabla u\|_{L^2(\Gamma_*)} \|\theta_h^* - I_h \theta_h^*\|_{L^2(\Gamma_*)} . 
    \end{aligned}\label{est3}
  \end{equation}
  First, by Theorem \ref{thm_highorder} with $l=1$, and $|x'|\sim (\kappa h)^{\frac{1}{2-\alpha}}$ on $\Gamma_*$, we have 
  \begin{align}\label{est4}
    \|\nabla u\|_{L^2(\Gamma_*)} \lesssim \|\nabla u\|_{L^\infty(\Gamma_*)} |\Gamma_*|^{\frac{1}{2}} \lesssim \big(\vep + (\kappa h)^{\frac{2}{2-\alpha}}\big)^{-1}(\kappa h)^{\frac{3/2}{2-\alpha}} \lesssim h^{-\frac{1/2}{2-\alpha}}.
  \end{align}
  Denote the mesh size on $\Gamma_*$ by $h_* \sim (\kappa h)^{\frac{\alpha}{2-\alpha}}h \sim h^{\frac{2}{2-\alpha}}$. The interpolation estimate yields
  \begin{equation}\label{est5}
    \begin{aligned}
      \|\theta_h^* - I_h \theta_h^*\|_{L^2(\Gamma_*)} 
      &\lesssim |\Gamma_*|^{\frac{1}{2}}\|\theta_h^* - I_h \theta_h^*\|_{L^\infty(\Gamma_*)} \\
      &\lesssim (\kappa h)^{\frac{3/2}{2-\alpha}} h_*^{p+1} \abs{\theta_h^*}_{W^{p+1,\infty}(\Gamma_*)} \\
      &\lesssim h^{\frac{2p+7/2}{2-\alpha}}\|\nabla^{p+1} \theta_h^* \|_{L^\infty(\Omega_*'\backslash\Omega_*)}.
    \end{aligned}
  \end{equation}
  Next, we estimate $\|\nabla^{p+1} \theta_h^* \|_{L^\infty(\Omega_*'\backslash\Omega_*)}$. 
  Since $\theta_h^* = I_q^* \theta_h= (\theta_h|_{\Gamma_1}) \tilde{v}_q + (\theta_h|_{\Gamma_2}) (1-\tilde{v}_q)$, with $\theta_h|_{\Gamma_1}$ and $\theta_h|_{\Gamma_2}$ being two constants, it follows that 
  \begin{align*}
    \|\nabla^{p+1} \theta_h^* \|_{L^\infty(\Omega_*'\backslash\Omega_*)} = \abs{\theta_h|_{\Gamma_1} - \theta_h|_{\Gamma_2}} \| \nabla^{p+1} \tilde{v}_q \|_{L^\infty(\Omega_*'\backslash\Omega_*)}.
  \end{align*}
  By the Newton--Leibnitz rule, we have
  \begin{align*}
      \left| \theta_h|_{\Gamma_1} - \theta_h|_{\Gamma_2} \right| = \bigg|\int_{\phi_2(x')}^{\phi_1(x')} \frac{\partial\theta_h}{\partial x_n} \diff x_n \bigg| 
  \end{align*}
  and therefore, by integrating this relation in the region $|x'|\le 1$ and using that $\abs{\theta_h|_{\Gamma_1} - \theta_h|_{\Gamma_2}}$ is a constant, we obtain
  \begin{align}
      \left| \theta_h|_{\Gamma_1} - \theta_h|_{\Gamma_2} \right| \lesssim \|\nabla\theta_h\|_{L^1} \lesssim \|\nabla\theta_h\|_{L^2}.
  \end{align}
  On the other hand, for $x=(x',x_n) \in {\Omega_*'\backslash\Omega_*}$ we have $|x'|\sim h^{\frac{1}{2-\alpha}}$ and therefore, according to the first result \eqref{est:v1} in Proposition \ref{thm:v1est}, 
  \begin{align*}
    \big|\nabla^{p+1} \tilde{v}_q(x) \big| \lesssim \delta(x')^{-\frac{p+2}{2}} 
    \lesssim h^{-\frac{p+2}{2-\alpha}} \;\; \mbox{for } x \in {\Omega_*'\backslash\Omega_*}.
  \end{align*}
  Hence, we obtain
  \begin{align*}
    \|\nabla^{p+1} \theta_h^* \|_{L^\infty(\Omega_*'\backslash\Omega_*)} \lesssim h^{-\frac{p+2}{2-\alpha}} \|\nabla\theta_h\|_{L^2},
  \end{align*}
  which together with \eqref{est4}--\eqref{est5} gives
  \begin{align}\label{est6}
    \|\nabla u\|_{L^2(\Gamma_*)} \|\theta_h^* - I_h \theta_h^*\|_{L^2(\Gamma_*)} \lesssim h^{\frac{p+1}{2-\alpha}} \|\nabla \theta_h\|_{L^2}.
  \end{align}
  Since $\alpha > \alpha_{\rm min} \geq 1-\frac{1}{p}$ implies $\frac{p+1}{2-\alpha}\ge p$, it follows from \eqref{est3} and \eqref{est6} that 
  \begin{align}\label{est:errh1}
    \|\nabla(u-u_h)\|_{L^2} \lesssim h^p.
  \end{align}
  This proves optimal-order convergence of the error in the $H^1$ semi-norm. In order to improve this result to the full $H^1$ norm, it remains to estimate $\|u-u_h\|_{L^2}$. 
  
  First, for $x = (x',x_n) \in \Omega_*$ and $v \in H^1(\Omega_*)$, the application of the Newton--Leibnitz rule leads to 
  \begin{align}\label{eq:nl}
    v(x', x_n) = v(x', \phi_2(x')) + \int_{\phi_2(x')}^{x_n} \frac{\partial v(x',z)}{\partial z} \diff z.
  \end{align}
  Since $u(x', \phi_2(x'))-u_h(x', \phi_2(x'))$ is constant on $\Gamma_2 \cap \partial \Omega_*$, it follows from Lemma \ref{lem:fc} that 
  \begin{align*}
    \big| u(x', \phi_2(x'))-u_h(x', \phi_2(x')) \big| \lesssim \|\nabla(u-u_h)\|_{L^2(\Omega_0)} + \|\varphi-I_h\varphi\|_{L^1(\Gamma)} \lesssim h^p.
  \end{align*}
  Therefore, choosing $v = u-u_h$ in \eqref{eq:nl} yields the following result:
  \begin{align*}
    &\quad\;\|u-u_h\|_{L^2(\Omega_*)}^2 \\
    &\lesssim \|u-u_h\|_{L^\infty(\Gamma_2\cap\partial\Omega_*)}^2\int_{\Omega_*} \diff x + \int_{|x'| < R_*} \diff x' \int_{\phi_2(x')}^{\phi_1(x')}  \Big|\int_{\phi_2(x')}^{x_n} \frac{\partial v(x',z)}{\partial z} \diff z\Big|^2 \diff x_n  \\
    &\lesssim h^{2p} \int_{\Omega_*} \diff x + \int_{|x'| < R_*} \diff x' \int_{\phi_2(x')}^{\phi_1(x')} \bigg(\int_{\phi_2(x')}^{x_n} \bigg|\frac{\partial (u-u_h)}{\partial z}\bigg|^2 \diff z \int_{\phi_2(x')}^{x_n} \diff z\bigg) \diff x_n \\
    &\lesssim h^{2p} + \|\nabla(u-u_h)\|_{L^2(\Omega_*)}^2 \\
    &\lesssim h^{2p},
  \end{align*}
In the subdomain $\Omega_1\backslash\overline\Omega_*$, the same argument shows that (since $u_h\in H^1(\Omega_*)$ and $u_h\in H^1(\Omega_*^c)$, but $u_h\notin H^1(\Omega_1)$, we need to establish the estimates in $\Omega_*$ and $\Omega_1\backslash\overline\Omega_*$, separately)
\begin{align*}
\|u-u_h\|_{L^2(\Omega_1\backslash\overline\Omega_*)}^2 \lesssim h^{2p}.
\end{align*}
Furthermore, the estimate of $\|u-u_h\|_{L^2(\Omega_0)}$ follows by using $u-u_h = (u-I_h u) + (I_h u -u_h)$ and the Poincar{\'e} inequality (because $I_h u -u_h=0$ on $\partial D\subset\partial\Omega_0$). 
This proves the error bound in the $L^2$ norm, i.e.,
\begin{align*}
\|u-u_h\|_{L^2} \lesssim h^{p}.
\end{align*}
This, together with the $H^1$ semi-norm bound in \eqref{est:errh1}, completes the proof of Theorem \ref{thm:error}.
\end{proof}

At the end of this section, we present the main proof of finite element error estimate for the linear elasticity problem.
\begin{proof}[Proof of Theorem \ref{thm:error} {for linear elasticity problem}] 
  We begin with the error equation, which is derived by the difference between FEM \eqref{FEM:elas} and the variational formulation of \eqref{eq:model1}, and tested with functions $v_h \in L^2(\Omega) \cap H^1(\Omega_*) \cap H^1(\Omega_*^c)$ satisfying $v_h|_{\partial D}=0$:
  \begin{equation}\label{eq:error}
    \begin{aligned}
      &\;\int_{\Omega_*^c \cup \Omega_*} \big( \lambda J(x) (\nabla\cdot \theta_h)(\nabla\cdot v_h) + 2\mu (A_h \nabla \theta_h, \nabla v_h) \big) \rd x \\
      = &\;\int_{\Omega_*^c \cup \Omega_*} \big( \lambda (1-J(x)) (\nabla \cdot u)(\nabla\cdot v_h) + 2\mu (e(u) - A_h \nabla u, \nabla v_h) \big) \rd x \\
      &\;\quad +\int_{\Omega_*^c \cup \Omega_*} \big( \lambda J(x) (\nabla \cdot \rho)(\nabla\cdot v_h) + 2\mu (A_h \nabla \rho, \nabla v_h) \big) \rd x \\
      &\;\quad -\int_{\Gamma_*} \big(\lambda (\nabla \cdot u)\vec{n} + 2\mu e(u) \vec{n} \big) \cdot \jump{v_h} \rd s,
    \end{aligned}
  \end{equation}
  where $\theta_h=u_h - I_h u \in S_h^0$ and $\rho=u-I_h u$. By \eqref{est:iso}, it is easy to verify that 
  \begin{align}\label{est:JA}
    \|1-J(x)\|_{L^\infty(\Omega)} \lesssim h^p 
    \quad\mbox{and}\quad
    |A_h \nabla v - e(v)| \lesssim h^p |\nabla v| \;\; \text{in }\Omega.
  \end{align}
  Letting $v_h=\theta_h$ in \eqref{eq:error} gives
  \begin{equation}\label{est:err2}
    \begin{aligned}
      \|\nabla \cdot \theta_h\|_{L^2}^2 + \|\nabla \theta_h\|_{L^2}^2 &\lesssim h^{2p} \big(\|\nabla \cdot u\|_{L^2}^2 + \|\nabla u\|_{L^2}^2 \big) + \|\nabla \cdot \rho\|_{L^2}^2 + \|\nabla \rho\|_{L^2}^2  \\
      & \;\quad + \big( \|\nabla\cdot u\|_{L^2(\Gamma_*)} + \|e(u)\|_{L^2(\Gamma_*)} \big) \big\|\theta_h^* - I_h \theta_h^* \big\|_{L^2(\Gamma_*)},
    \end{aligned}
  \end{equation}
  It remains to analyze the last term in \eqref{est:err2}. It is quite similar to that for the perfect conductivity problem, we only need to consider the Case 2 in three dimensions. The estimates \eqref{est4} and \eqref{est5} still hold for this case, with the only exception being the analysis of $\|\nabla^{p+1}\theta_h^*\|_{L^\infty(\Omega'_* \backslash \Omega_*)}$. We begin our analysis from this point. First, we recall that 
  \begin{align*}
    \theta_h^* = I_q^* \theta_h
    = \sum_{j=1}^2 \sum_{l=1}^{6} (\theta_h)_{l,j} \Lambda_{l,j}^{(q)}(x) = \sum_{j=1}^2 \sum_{l=1}^{6} (\theta_h)_{l,j} \psi_l(x) \sum_{m=1}^q \mathbf{v}_{l,j}^{(m)}(x),
  \end{align*}
  where $(\theta_h)_{l,j}$ are free constants of $\theta_h$ on $\Gamma_j$ and $\Lambda_{l,j}^{(q)}$ are the auxiliary functions defined in \eqref{eq:Iq-elas}. 
  Similar to the proof of Lemma \ref{lem:fc}, noting $\theta_h = u_h - I_h u = 0$ on $\Gamma=\partial D$, we have
  \[
    |(\theta_h)_{l,j}| \lesssim \|\nabla\theta_h\|_{L^1} \lesssim \|\nabla\theta_h\|_{L^2}.
  \]
  By using the estimates for $\mathbf{v}_{l,j}^{(m)}$ in \cite[Sections 2 and 4]{DLTZ}, we have
  \begin{equation}\label{est:th*-1}
    |\nabla^t \mathbf{v}_{l,j}^{(m)}(x)| \lesssim \delta(x')^{-\frac{t+1}{2}} \lesssim h^{-\frac{t+1}{2-\alpha}} \quad\text{for all }t \geq 1 \text{ and } x=(x',x_n) \in \Omega_*'\backslash\Omega_*.
  \end{equation}
  Moreover, noting $|\nabla^t \psi_l(x',\phi_j(x'))| \lesssim |\nabla^t \phi_j(x')| \lesssim 1+|x'|^{2+\gamma-t} \lesssim 1+h^{\frac{2+\gamma-t}{2-\alpha}}$, we get
  \begin{align}
    \| \nabla^{p+1} \theta_h^* \|_{L^\infty(\Omega'_* \backslash \Omega_*)} \lesssim (h^{-\frac{p+2}{2-\alpha}}+h^{-\frac{p-\gamma}{2-\alpha}}) \| \nabla\theta_h \|_{L^2} \lesssim  h^{-\frac{p+2}{2-\alpha}} \| \nabla\theta_h \|_{L^2},
  \end{align}
  which gives
  \begin{equation*}
    \begin{aligned}
      \big( \|\nabla\cdot u\|_{L^2(\Gamma_*)} + \|e(u)\|_{L^2(\Gamma_*)} \big) \big\|\theta_h^* - I_h \theta_h^* \big\|_{L^2(\Gamma_*)} 
      \lesssim h^{\frac{p+1}{2-\alpha}} \| \nabla \theta_h \|_{L^2} 
      \lesssim h^p \| \nabla \theta_h \|_{L^2},
    \end{aligned}
  \end{equation*}
  Then, by \eqref{est:err2}, we arrive at
  \begin{align}
    \|\nabla \cdot (u-u_h) \|_{L^2} + \|\nabla (u-u_h)\|_{L^2} \lesssim h^p. \label{est:grad-err}
  \end{align}
  The $L^2$-estimate is obtained similarly to the perfect conductivity problem; the details are omitted for brevity, and this completes the proof of Theorem \ref{thm:error} for the linear elasticity problem.
\end{proof}

\section{Numerical tests}\label{sec:tests}
In this section, we provide some numerical tests to validate the theoretical results. All the codes are written by Firedrake (an open-source finite element package \cite{hkm+23,bmh+16}). 
%
The finite element solutions $\hat u_{h_i}$ are solved under a series of graded meshes with decreasing sizes $h \in \{h_i : h_i > h_{i+1}\}$. The relative $H^j$-errors are defined by
\[
\mbox{Relative $H^j$-error} = \frac{\|\hat u_{h_i} - \hat u_{h_{i+1}}\|_{H^j(\Omega)}}{\| \hat u_{h_{i+1}} \|_{H^j{(\Omega)}}}, \quad j=0,1.
\]

\subsection{Perfect conductivity}
We first give some numerical tests for the perfect conductivity problem in both two and three dimensions. 
The right-hand-side function is given by the linear potential:
\begin{align*}
  \varphi(x) = 
  \begin{cases}
    x_2 - x_1,       & n=2, \\
    x_3 - x_2 - x_1, & n=3.
  \end{cases}
\end{align*}

\begin{example}[Convergence rates] 
	Let $D=\mathcal{B}_3$ denote a disk with radius $r=3$ and centered at origin $(0,0)$, and let the inclusions $D_1$ and $D_2$ be two disks with radius $r=1$ and centered at $(0,r+\vep/2)$ and $(0,-r-\vep/2)$, respectively. 
	The mesh sizes are $h \in \{2^{-2}, 2^{-3}, \cdots, 2^{-6}\}$. 
	We take $\varepsilon=0.1$ and $\varepsilon=10^{-5}$, respectively. The refinement parameters are chosen as $\kappa=1$, with $\alpha$ taking various values that satisfy the condition stated in Theorem \ref{thm:error}. 
  As discussed in previous section, the graded meshes are constructed according to Case 1 when $\varepsilon=0.1$, and to Case 2 when $\varepsilon=10^{-5}$. 
  Let the finite element orders be $p=1,\cdots,5$. 
	The order $q$ of interpolation in the narrow region is chosen as discussed in Remark \ref{rmk:pql}, i.e.,
	we set $q=1$ if $p \leq 2$ and $q=2$ if $3 \leq p \leq 5$.

  The log-log plots of the relative $H^1$-errors of finite element solutions are presented in Figure \ref{fig:errors_p5}. 
  It is shown that the $H^1$-errors behave as $O(h^p)$ when $\varepsilon=0.1$ and $\alpha=1$, which coincides with the optimal rate. 
  For the case of $\varepsilon=10^{-5}$ and $\alpha = 1.35$ satisfying that $\alpha \in (2-\frac{7}{2p}, \frac{3}{2})$, 
  the convergence rates are still optimal in this case. 
  These results show that the proposed high-order FEM is uniformly optimally convergent for arbitrary small $\varepsilon$, and it can be also observed that the convergence errors are independent of $\varepsilon$.

	\begin{figure}
		\centering
		\begin{minipage}[c]{0.48\textwidth}
			\includegraphics[width=\textwidth]{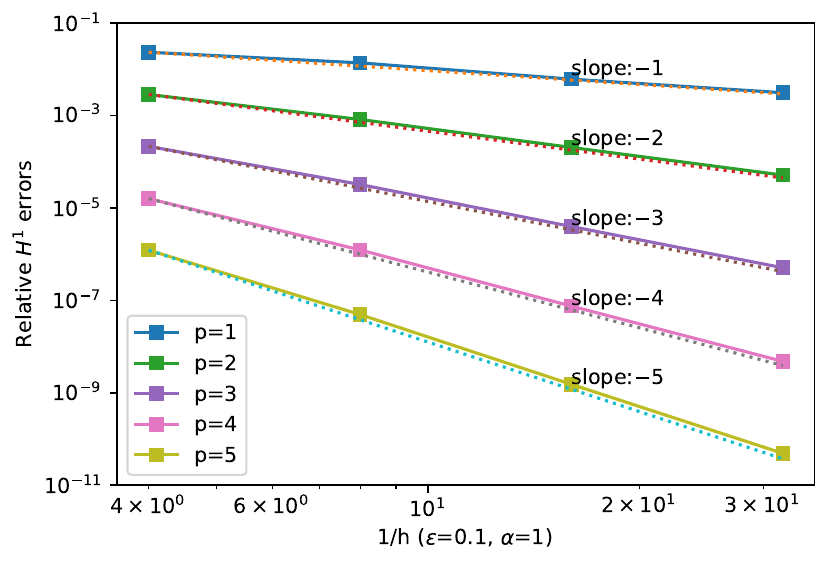}
		\end{minipage}
		\begin{minipage}[c]{0.48\textwidth}
			\includegraphics[width=\textwidth]{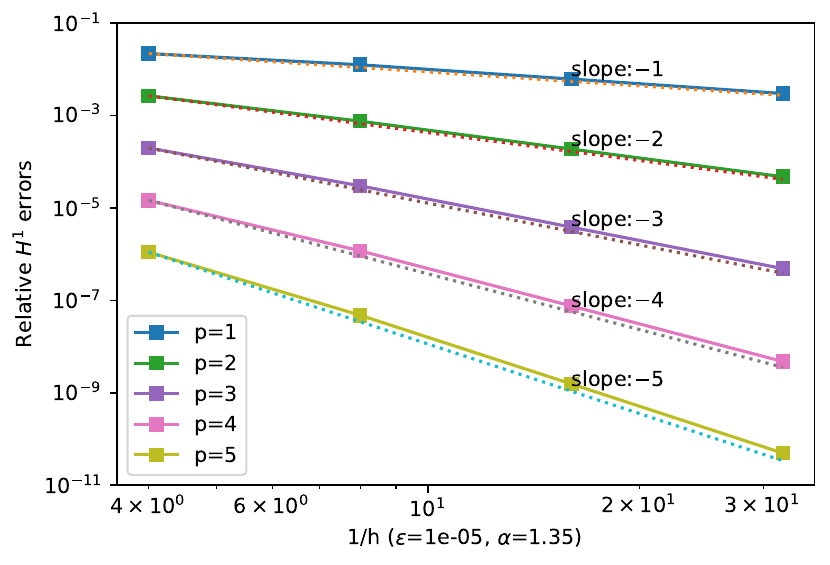}
		\end{minipage}
		\caption{Left: $\vep=0.1$ and $\alpha =1$; Right: $\vep=10^{-5}$ and $\alpha=1.35$.}
		\label{fig:errors_p5}
	\end{figure}

  In Figure \ref{fig:compare}, we plot the errors for the case that only linear interpolation is used in $\Omega_*$, i.e., $q=1$ in \eqref{eq:Ih*}. 
  It can be observed that for high-order elements $(p > 3)$, the choice $q = 1$ fails to guarantee optimal convergence rates, even when mesh parameter $\alpha$ reaches its theoretical supremum $1+\frac{1}{n}=1.5$.

	\begin{figure}
		\centering
		\begin{minipage}[c]{0.48\textwidth}
			\includegraphics[width=\textwidth]{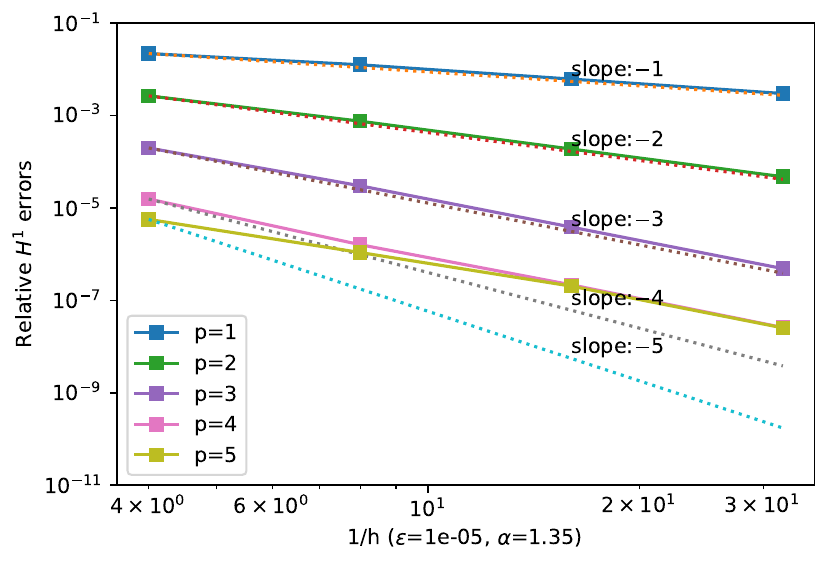}
		\end{minipage}
		\begin{minipage}[c]{0.48\textwidth}
			\includegraphics[width=\textwidth]{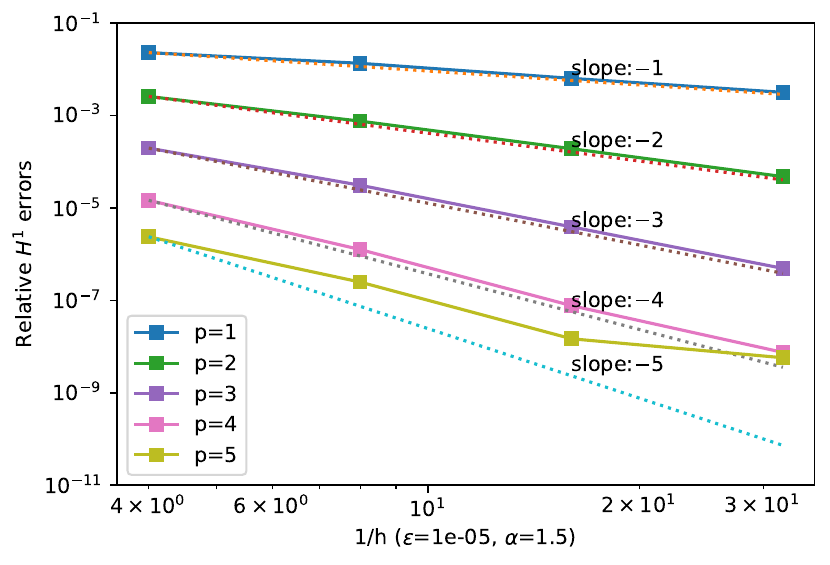}
		\end{minipage}
		\caption{Fixed $q=1$. Left: $\vep=10^{-5}$ and $\alpha=1.3$; Right: $\vep=10^{-5}$ and $\alpha=1.5$.}
		\label{fig:compare}
	\end{figure}

  \begin{table}[htbp]
    \centering
    \caption{Condition numbers for different orders $p$ and mesh sizes $h$}
    \label{tab:condn}
    \begin{tabular}{c c c c c}
    \toprule
    & $h=2^{-2}$ & $h=2^{-3}$ & $h=2^{-4}$ & $h=2^{-5}$ \\
    \midrule
    $p=1$ & $4.0844\times 10^{4}$ & $1.6733\times 10^{5}$ & $7.5867\times 10^{5}$ & $3.8597\times 10^{6}$ \\
    $p=2$ & $1.8387\times 10^{5}$ & $7.9454\times 10^{5}$ & $4.1513\times 10^{6}$ & $2.7225\times 10^{7}$ \\
    $p=3$ & $4.4647\times 10^{5}$ & $2.0650\times 10^{6}$ & $1.2322\times 10^{7}$ & $9.4483\times 10^{7}$ \\
    $p=4$ & $8.5183\times 10^{5}$ & $4.2318\times 10^{6}$ & $2.7989\times 10^{7}$ & $2.3620\times 10^{8}$ \\
    \bottomrule
    \end{tabular}
  \end{table}
  \begin{figure}[thbp]
    \centering
    \includegraphics[height=4cm]{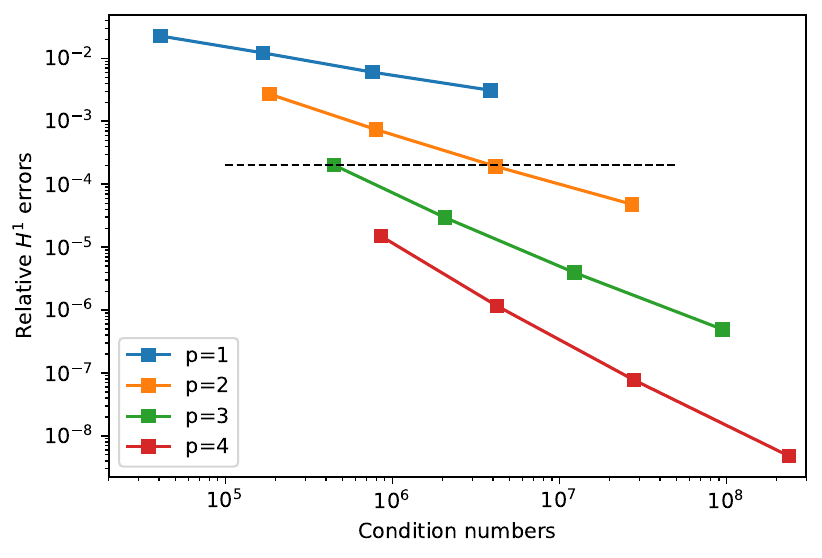}
    \caption{$H^1$ errors versus condition numbers: for each fixed $p$, the points (from left to right) correspond to successive mesh refinements, i.e., $h=2^{-2},h^{-3},h^{-4},h^{-5}$.}
    \label{fig:condn}
  \end{figure}
  Finally, we report the condition numbers for this example in Table \ref{tab:condn}, and plot the relative $H^1$ errors versus the condition numbers in Figure \ref{fig:condn}. It is observed that the condition number increases as $h$ decreases and $p$ increases. Moreover, to achieve the same level of accuracy, higher-order methods yield significantly smaller condition numbers than lower-order methods. The design of efficient preconditioners for the proposed FEM, which combines high-order elements, graded meshes, and auxiliary basis functions, remains an open and challenging problem that we plan to investigate in future work.

\end{example}

\begin{example}
  In this example, we consider an asymmetric configuration in which the two inclusions are a circle and an ellipse, respectively. Figure \ref{fig:asym} displays the convergence rates and the numerical solution. The results demonstrate that optimal convergence rates are achieved even for $\varepsilon=10^{-5}$, indicating that the proposed high-order FEM remains effective for asymmetric geometries.
  \begin{figure}[thbp]
    \begin{minipage}[c]{0.48\textwidth}
      \includegraphics[width=\textwidth]{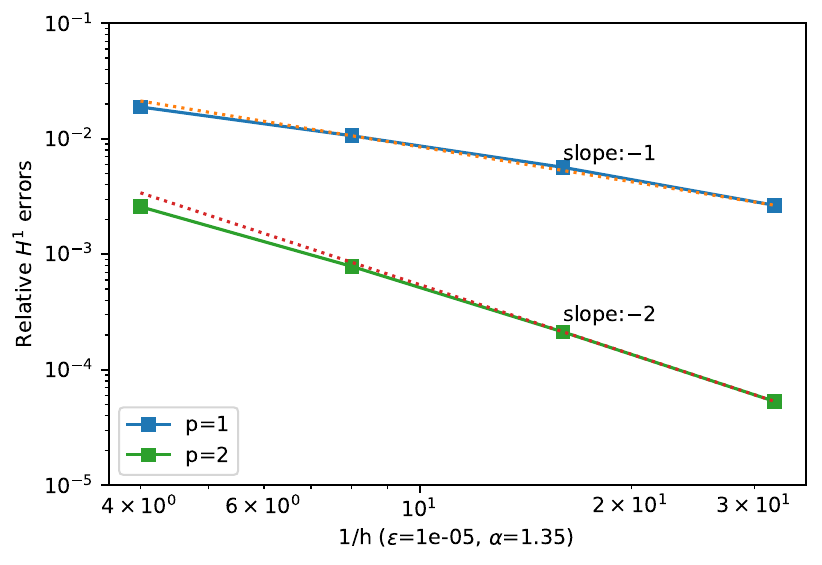}  
    \end{minipage}
    \begin{minipage}[c]{0.48\textwidth}
      \includegraphics[width=\textwidth]{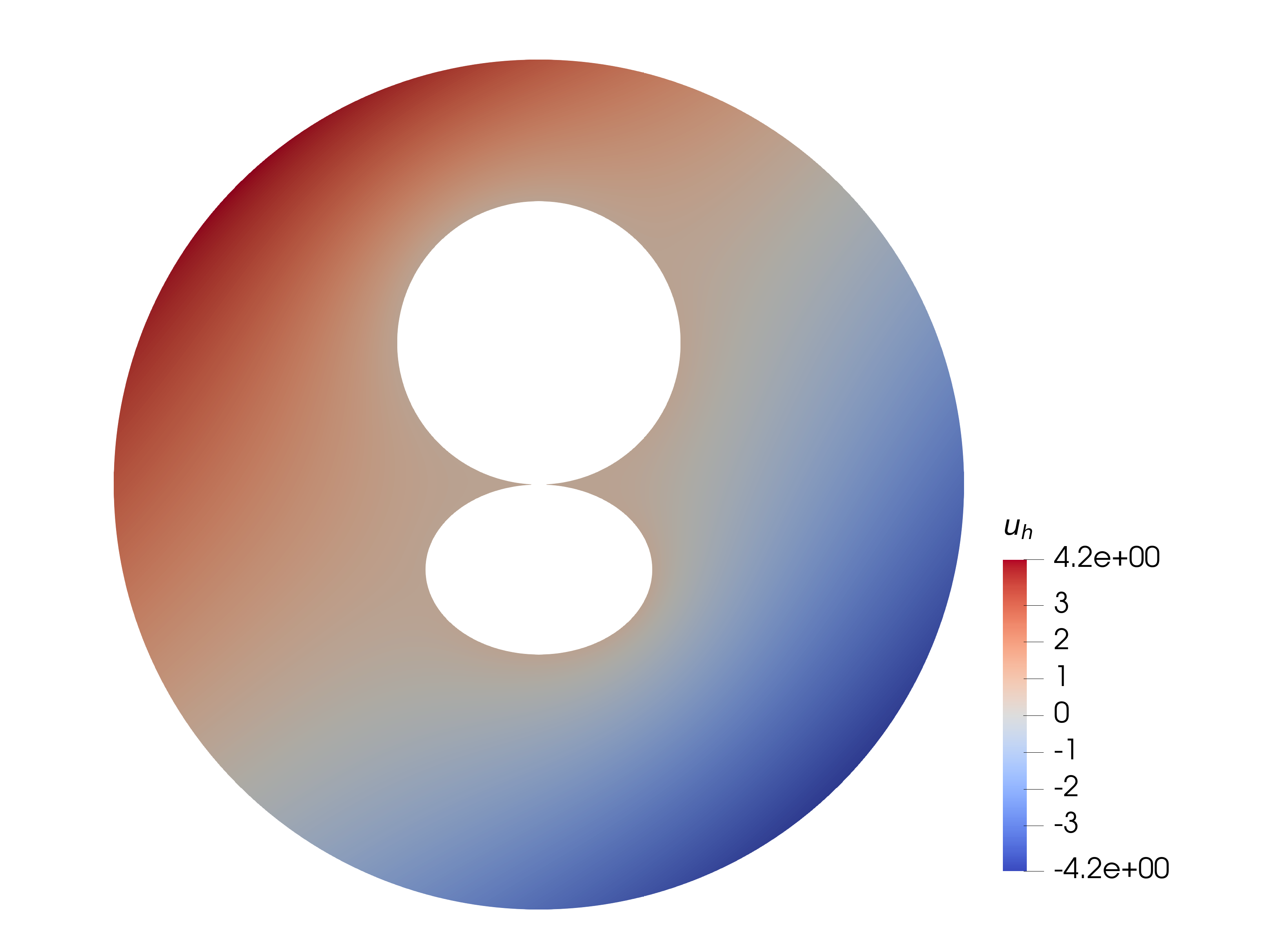}
    \end{minipage}
		\caption{Left: convergence rates; Right: numerical solution when $p=2$ and $h=2^{-6}$.}
		\label{fig:asym}
	\end{figure}
\end{example}

\begin{example}[Blow-up of gradient of solutions] 
  Let $D=\mathcal{B}_{2.4}$ be a disk, and let the inclusions in $D$ be two ellipses with long and short axes $(1,0.8)$. 
  We take $\varepsilon=0.1$ and $\varepsilon=10^{-5}$, respectively. 
  In Figure \ref{fig:exm2D:grad}, we plot the gradients of the finite element solutions, 
  which clearly shows the blow-up behaviour for small $\varepsilon$ near the close-to-touching point. 
  \begin{figure}[t]
    \centering
    \includegraphics[width=\textwidth]{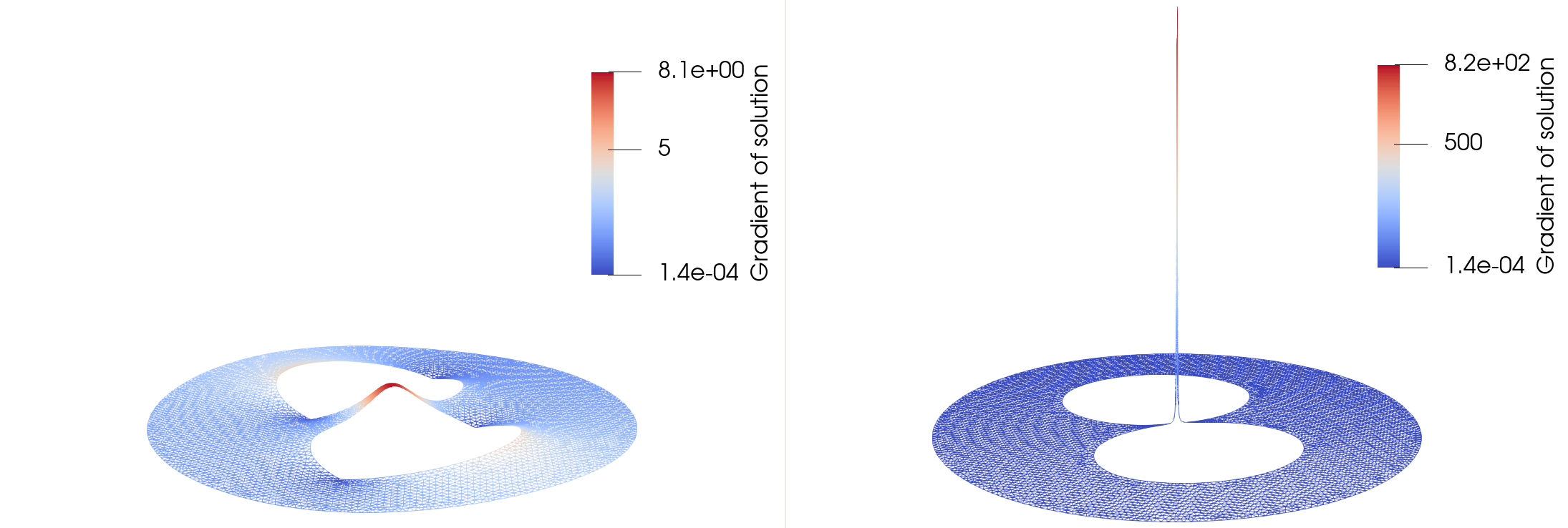}
    \caption{The gradients of solutions $|\nabla \hat u_h|$ for $\varepsilon=0.1$ (left graph) and $\varepsilon=10^{-5}$ (right graph) with $p=3$ and $h=1/16$.}
    \label{fig:exm2D:grad}
  \end{figure}
\end{example}

\begin{example}[Ellipsoidal inclusions in 3D] 
  Let the domain $D=\mathcal{B}_{2.1}$ be a ball, and let the inclusions be two ellipsoidal inclusions with axes $(1,1,0.7)$. The mesh parameters are $\kappa=1$ and $\alpha=1.33$. 
  Figure \ref{fig:exm3D:mesh} shows the graded mesh for $h=0.2$ and $\varepsilon=10^{-5}$.
  The relative $H^1$-errors are plotted in Figure \ref{fig:exm3D}, which indicates the optimal rate in $H^1$-norm is $O(N^{-p/3})$ for $p=1,2$. 

  \begin{figure}[t]
    \centering
    \begin{minipage}[c]{0.4\textwidth}
      \includegraphics[width=\textwidth]{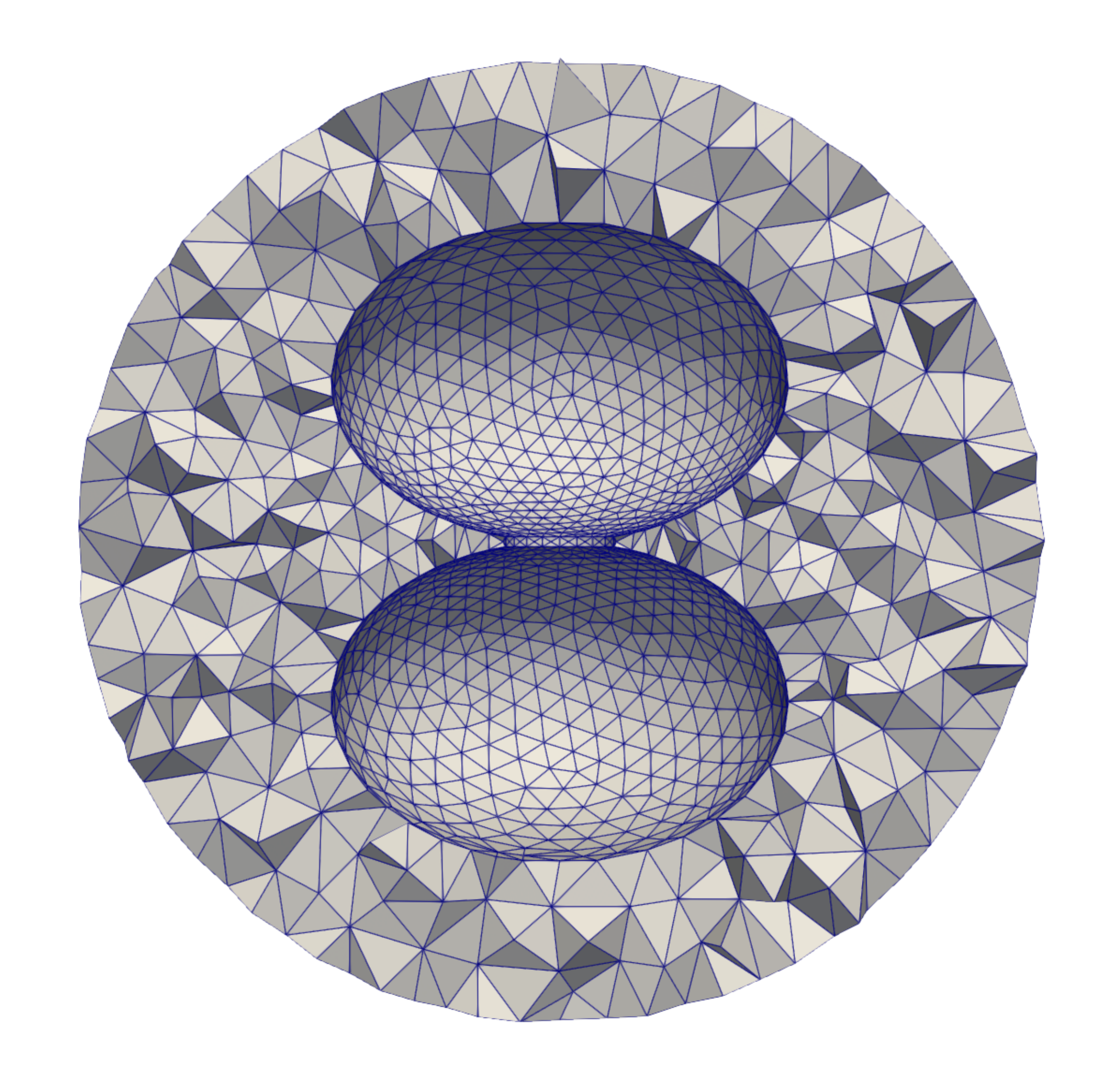}
    \end{minipage}\hspace{1cm}
    \begin{minipage}[c]{0.4\textwidth}
      \includegraphics[width=\textwidth]{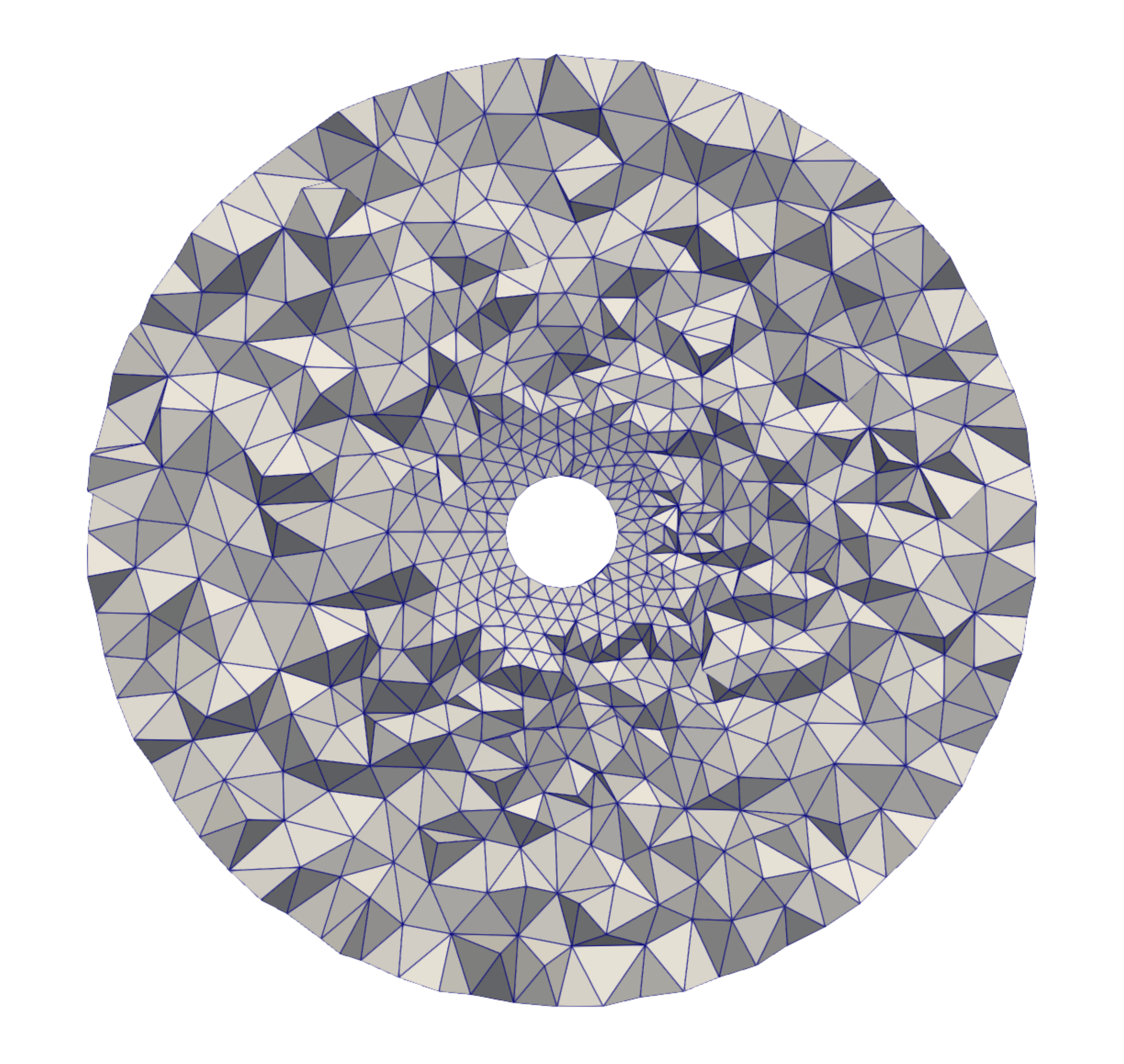}
    \end{minipage}
    \caption{The slices of the graded mesh near $x=0$ (left) and $z=0$ (right).}
    \label{fig:exm3D:mesh}
  \end{figure}
  \begin{figure}[t]
    \centering
    \includegraphics[width=0.8\textwidth]{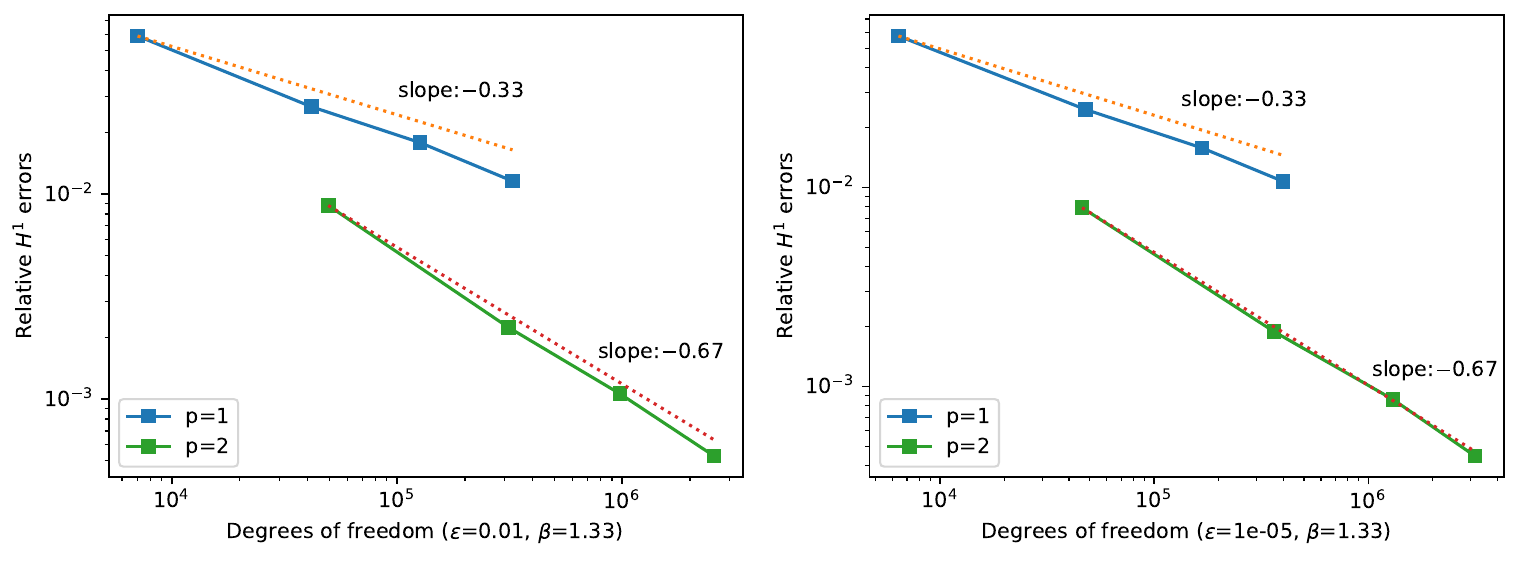}
    \caption{The relative $H^1$-errors for $\varepsilon=0.01,10^{-5}$.}
    \label{fig:exm3D}
  \end{figure}
\end{example}

\subsection{Linear elasticity}
Next, we give some numerical tests for the linear elasticity problem. 

\begin{example}[Elliptical inclusions in 2D]
  Let the domain $D=\mathcal{B}_{2.4}$ be a disk. The inclusions are both ellipses with axes $(1,0.8)$. 
  We choose $\varepsilon=0.1$ and $\varepsilon=10^{-5}$. Let $\kappa=1$ and $p=1,2,3$. 
  The vector-valued right-hand-side function is given by
  \[
    \varphi(x) = (x_2 - x_1, x_2 + x_1)^\text{T}.
  \]
  Figure \ref{fig:high-order} presents the relative $H^1$-errors for different values of $\alpha$. 
  For $\alpha=1.5$, the proposed method converges at optimal rate $O(h^p)$ for $p \leq 2$, and again, the errors are independent of $\vep$.  
  While for $p = 3$, the convergent rate is slightly worse than $O(h^3)$ when $\vep$ is small. 
  For $\alpha=1$, only the case of $p=1$ reaches the optimal convergence rate. This indicates that the choice of $\alpha$ does affect the convergence result for proposed high-order FEM. 
  These findings are in agreement with our theoretical analysis in Theorem \ref{thm:error}.
  \begin{figure}[t]
    \centering
    \begin{minipage}[c]{0.9\textwidth}
      \centering
		  \includegraphics[width=\textwidth]{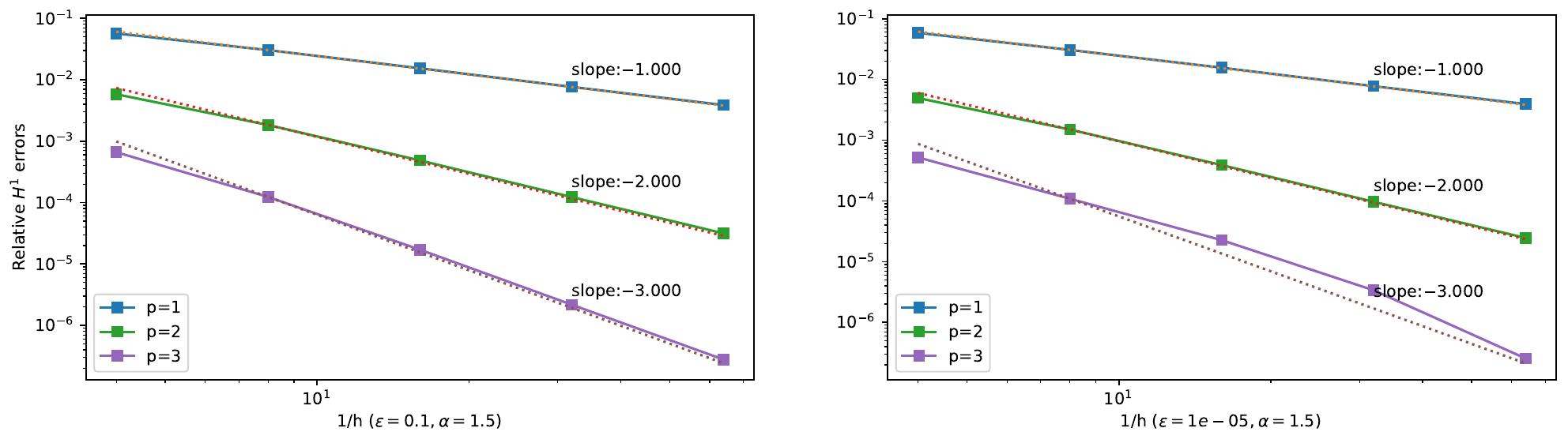}
    \end{minipage}\,
    \begin{minipage}[c]{0.9\textwidth}
      \centering
		  \includegraphics[width=\textwidth]{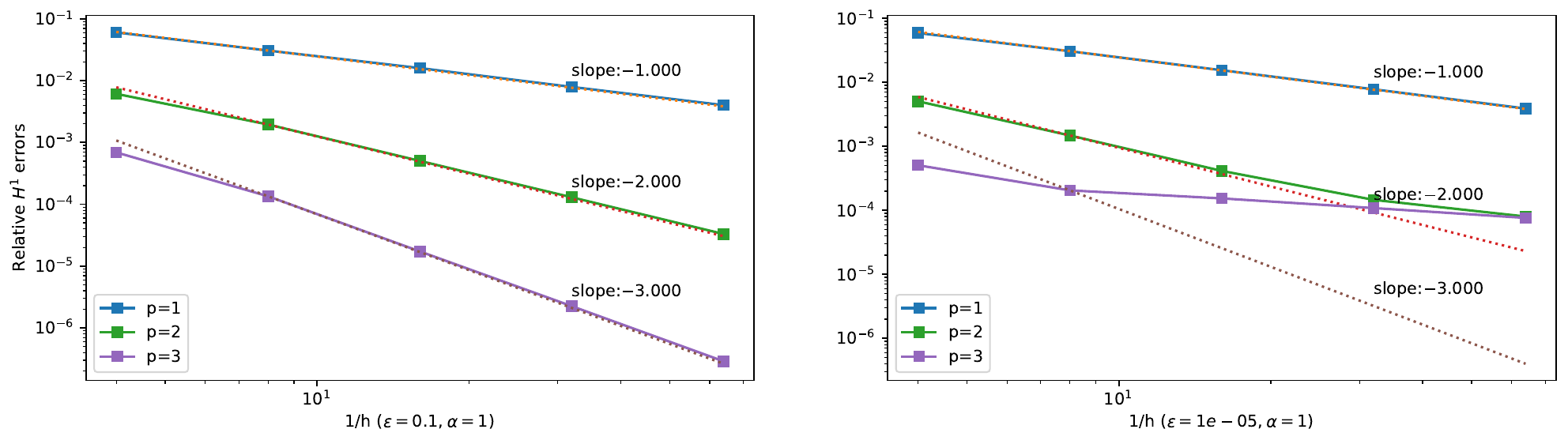}
    \end{minipage}
    \caption{The relative $H^1$-errors for $\varepsilon=0.1,10^{-5}$ and $\alpha=1,1.5$, respectively.}
    \label{fig:high-order}
  \end{figure}
\end{example}

\section{Proof of Proposition \ref{thm:v1est}}\label{sec:highorderest}
In this section, we present the proof of Proposition \ref{thm:v1est}, which directly implies Theorem \ref{thm_highorder}. The latter is used in designing a high-order convergent FEM for solving \eqref{PDE}. 

\begin{proof}[Proof of Proposition \ref{thm:v1est}]
First, we give some results for the auxiliary function $\tilde v_l = \sum_{k=1}^{l} \bar v_k$ and estimate the term $\nabla^l \tilde v_q(x)$ in the first result \eqref{est:v1} of Proposition \ref{thm:v1est}.
For the function $\bar v_1$ defined in \eqref{def-bar-v1}, direct calculation yields the following estimates for $1\le j\le l-1$ and $x \in \Omega_{1/2}$: 
\begin{equation}\label{nablagjdsv1}
  |\bar{ v}_{1}|\lesssim 1,\quad |\nabla_{x'}^j \bar{ v}_{1}|\lesssim \delta(x')^{-j/2},\quad |\nabla_{x'}^{j-1}\partial_{x_n} \bar{ v}_{1}|\lesssim \delta(x')^{-(j+1)/2},\quad \partial_{x_n}^2 \bar{ v}_{1}=0.
\end{equation}
Using \eqref{defgreenfunction3d} and some direct calculations, we have
\begin{align}\label{estgreem}
  \begin{split}
    &|G(y,x_n)|\lesssim \delta(x'), \quad |\nabla_{x'}G(y,x_n)|\lesssim \delta(x')^{1/2},\quad|\partial_{x_n}G(y,x_n)|\lesssim 1, \quad\text{and}\\
    &|\nabla_{x'}^jG(y,x_n)|\lesssim \delta(x')^{(2-j)/2},\quad 1_{y\neq x_n}\cdot\partial_{x_n}^jG(y,x_n)=0\quad\mbox{for}~2\le j\le l+1.
  \end{split}
\end{align}
By mathematical induction, the relationships \eqref{eq:vk}, \eqref{scfest}, and the estimates \eqref{nablagjdsv1}--\eqref{estgreem}, we can prove the following estimates (details are omitted): for $k=1,2,\dots,l$,
\begin{align}
  |{\bar v}_k(x',x_n)| \lesssim \delta(x')^{k-1}, \quad |f_{k}(x',x_n)|=|\Delta_{x'}\bar{v}_k|\lesssim \delta(x')^{k-2} \quad \text{in}~\Omega_{1/2},\label{scfkdgj} 
\end{align}
and for $j\ge 0$,
\begin{align}\label{scxydest2}
  &\begin{aligned}
    &|\nabla_{x'}^j\partial_{x_n}^s{\bar v}_k(x',x_n)|\lesssim \delta(x')^{\frac{2k-2s-j-2}{2}}  & & \text{for}~0\le s\le 2k-1, \\
    &\partial_{x_n}^{s}{\bar v}_k(x',x_n)=0  & & \text{for}~s\ge 2k, \\
    &|\nabla^s f_{k}(x',x_n)|\lesssim \delta(x')^{k-2-s}  & & \text{for}~1 \leq s \leq k-1
  \end{aligned}\quad \text{in}~\Omega_{1/2}.
\end{align}
Using \eqref{scxydest2}, we can derive that
\begin{equation}\label{zx2dv1l1}
  \begin{aligned}
    |\nabla^l{\bar v}_{k}(x)| &\lesssim \delta(x')^{-\frac{l+1}{2}} \text{ for }k \leq l/2; \;\; |\nabla^l{\bar v}_{k}(x)| \lesssim \delta(x')^{k-l-1} \text{ for }k \geq {(l+1)}/{2}.
  \end{aligned}
\end{equation}
Since ${\tilde v}_q=\sum_{k=1}^q {\bar v}_k$, as defined in \eqref{tvl}, the first result of \eqref{zx2dv1l1} implies that 
\begin{align}\label{est:tvl}
|\nabla^l {\tilde v}_q(x',x_n)|\lesssim \delta(x')^{-\frac{l+1}{2}},\quad (x',x_n)\in \Omega_{1/4}\quad\mbox{for all}\,\,\, l,q\ge 1
\end{align}

Next, we consider the difference $w_l = v_1-\tilde v_l$. From \eqref{equ-v1} and \eqref{eq:tvl} we see that $w_l$ satisfies the following equation: 
\begin{equation}\label{equ_w}
  \Delta w_{l}=-\Delta \tilde{v}_{l}=f_{l} \;\;\mbox{in }\Omega_{1/2}; \quad w_l = 0\;\;\text{on }\Gamma^\pm_{1/2}.
\end{equation}
Then, we estimate the term $\nabla^l v_1(x)$ in \eqref{est:v1} by using the following lemma, which is the second result \eqref{est:v1-tv1} of Proposition \ref{thm:v1est}. Using the Sobolev embedding theorem, the estimate for $w_l$ in the following lemma holds. Due to its technical nature, the proof of this lemma is presented in SM1 in the supplemental material for interest readers. 
\begin{lemma}\label{propnew}
Under the conditions of Proposition \ref{thm:v1est}, the solution of \eqref{equ_w} satisfies the following estimate for $0 \leq s \leq l$: 
  \begin{align*}
    |\nabla^{s}{w}_{l}(x) | \lesssim \delta(x')^{l-s},\quad x=(x',x_n) \in \Omega_{1/4}.
  \end{align*}
\end{lemma}

Now, choosing $s=l$ in Lemma \ref{propnew} immediately yields $\|\nabla^{l}({v}_{1}-\tilde{v}_l)\|_{L^\infty(\Omega_{1/4})}\leq C$, indicating that $\tilde v_l$ captures all the singularities of $v_1$ up to the $l$-th order derivatives.
This, together with \eqref{est:tvl}, yields
  \begin{align*}
  |\nabla^l{v}_1(x',x_n)|\lesssim |\nabla^l {\tilde v}_l(x',x_n)|+C\lesssim \delta(x')^{-\frac{l+1}{2}},\quad (x',x_n)\in \Omega_{1/4}.
  \end{align*}
The proof of Proposition \ref{thm:v1est} is completed.
\end{proof}

%
%

\begin{remark}
It follows from \eqref{geo2} that, for $m\ge 2$, 
\begin{equation}\label{zbgj}
    \phi_{2}(x')\lesssim (\varepsilon+|x'|^2)\lesssim\delta(x'),\,\,|\partial_{x_i}\delta(x')|\lesssim|x'|\lesssim\delta(x')^{1/2},\,\,|\nabla_{x'}^m\phi_{2}(x')|,|\nabla_{x'}^m \delta(x')|\lesssim 1.
\end{equation}
Recalling that
  $$\bar{v}_{1}(x',x_n):=\frac{x_n-\phi_{2}(x')}{\delta(x')} \quad \text{for}~x\in\Omega_{1/2}.$$
A direct calculation yields, for $1\le i\le n-1$,
\begin{equation}\label{bd1}
    \partial_{x_i}\bar{v}_{1}=\frac{-\partial_{x_i}\phi_{2}(x')\delta(x')-\partial_{x_i}\delta(x')(x_n-\phi_{2}(x'))}{\delta(x')^2}.
\end{equation}
Thus, by \eqref{zbgj}, we obtain
\begin{equation*}
   |\bar{ v}_{1}|\lesssim 1,\quad   |\partial_{x_i}\bar{v}_{1}|\lesssim \delta(x')^{-1/2},
\end{equation*}
We now prove by induction that, for $j\ge1$,
\begin{equation}\label{gn1}
  \partial_{x_i}^j \bar{v}_{1}(x',x_n)=\frac{P_1^j(x')x_n+P_2^j(x')}{\delta(x')^{j+1}}, 
\end{equation}
where
\begin{equation}\label{dtgs1-1}
    P_\alpha^{j}(x')=\partial_{x_i}P_\alpha^{j-1}(x')\delta(x')-j\partial_{x_i}\delta(x')P_\alpha^{j-1}(x'), \quad \alpha=1,2.
\end{equation}
In particular,
\begin{equation*}
    P_\alpha^{2}(x')=\partial_{x_i}P_\alpha^{1}(x')\delta(x')-j\partial_{x_i}\delta(x')P_\alpha^{1}(x'),\quad \alpha=1,2,
\end{equation*}
where
\begin{equation}
    P_1^{1}=-\partial_{x_i}\delta(x'), \quad  P_2^{1}=\partial_{x_i}\delta(x')\phi_{2}(x')-\partial_{x_i}\phi_{2}(x')\delta(x').
\end{equation}
By \eqref{zbgj} again, we have
\begin{equation*}
   |\partial_{x_i}^lP_1^{2}|\lesssim \delta(x')^{1-l/2},\quad  |\partial_{x_i}^lP_2^{2}|\lesssim \delta(x')^{2-l/2},\quad l=0,1.
\end{equation*}
Combining this with the recurrence relation \eqref{dtgs1-1}, we obtain
\begin{equation*}
    |\partial_{x_i}^lP_1^j(x')|\lesssim \delta(x')^{(j-l)/2}, \quad |\partial_{x_i}^lP_2^j(x')|\lesssim\delta(x')^{(j-l)/2+1},\quad l=0,1.
\end{equation*}
Substituting this into \eqref{gn1} yields
\begin{equation}
    |\partial_{x_i}^j \bar{v}_{1}(x',x_n)|\lesssim\delta(x')^{j/2}.
\end{equation}
The estimates for the remaining derivatives in \eqref{nablagjdsv1} can be derived in a similar manner.
\end{remark}

\begin{remark}
We emphasize that our method can be directly extended to higher dimensions $(n\ge 4)$, giving the bound $ |\nabla^l u|\leq (\varepsilon+|x'|^2)^{-\frac{l+1}{2}}$. We omit the detailed derivation, as it can be readily reconstructed by interested readers.
\end{remark}

\appendix
\section{Proof of Lemma \ref{propnew}}\label{app:w_l}

Let $w_{l}$ be the solution to \eqref{equ_w}. In this section, we prove Lemma \ref{propnew} by estimating the high-order derivatives of $w_l = v_1-\tilde v_l$ in the following narrow region:
\begin{equation}\label{defraz}
    R(a,z'):= \Big\{(x',x_{n})\in\Omega_{\frac{1}{4}}: \phi_{2}(x')<x_{n}<\phi_{1}(x'),~|x'-z'|<a\delta(z') \Big\},\,\,
\mbox{for}~|z'|<\frac14,
\end{equation}
where $a>0$ is a constant. For the simplicity of notation, we write $R(z'):=R(1,z')$.

\begin{lemma}\label{w1infty}
Under the conditions of Proposition \ref{thm:v1est}, the following estimate holds for $z=(z',z_n)\in \Omega_{1/4}$ and $1\le s\le l$: 
	\begin{align}
		&\;\quad \|\nabla^{s}w_{l}\|_{L^{\infty}(R(\frac12,z'))} \notag \\
		&\lesssim \frac{1}{\delta(z')^{s}}\Big(\frac{1}{\delta(z')^{(n-2)/2}}\|\nabla w_{l}\|_{L^{2}(R(z'))}+\sum_{m=0}^{s-1} \delta(z')^{2+m} \|\nabla^{m} f_{l}\|_{L^{\infty}(R(z'))}\!\Big).\label{estmain}
	\end{align}
\end{lemma}
\begin{proof}
	Similarly as \cite[Step 2 in the proof of Proposition 1.7]{Li20}, we use the change of variables
	\begin{equation*}
		\begin{cases}
			y'=\frac{1}{\delta(z')}(x'-z'),\\
			y_{n}=\frac{x_{n}}{\delta(z')},
		\end{cases}
	\end{equation*} 
	to transform $R(z')$ into a cylinder $Q_{1}$ of unit size, where 
	$$Q_{r}:=\Big\{(y',y_{n})\in\mathbb{R}^{n} : \frac{1}{\delta(z')}\phi_{2}(z'+\delta(z')y')<y_{n}<\frac{1}{\delta(z')}\phi_{1}(z'+\delta(z')y'),\;|y'|<r\Big\},$$ 
  with top and bottom boundaries
	\begin{align*}
    \tilde{\Gamma}_{j}^{r} &= \Big\{(y',y_{n})\in\mathbb{R}^{n}: y_{n}=\tilde{\phi}_{j}(y'):=\frac{1}{\delta(z')}\phi_{j}(z'+\delta(z')y'),\;|y'|<r\Big\}, \quad j=1,2.
  \end{align*}
	For simplicity, we denote 
	$$W(y',y_{n})=w_{l}(z'+\delta(z')y',\delta(z')y_{n})\quad\mbox{and}\quad \tilde{V}_{l}(y',y_{n})=\tilde{v}_{l}(z'+\delta(z')y',\delta(z')y_{n}),$$
	then
	\begin{equation}\label{equ_W}
		\begin{cases}
			\Delta W=F_{l}:=-\Delta \tilde{V}_{l}&\mbox{in}~Q_1,\\
			W=0 &\mbox{on}~\tilde{\Gamma}_{1}^{1}\cup\tilde{\Gamma}_{2}^{1}.
		\end{cases}
	\end{equation}
	Since $\tilde{\phi}_{1}$ and $\tilde{\phi}_{2}$ are smooth, then for any $y_0\in\tilde{\Gamma}_{1}^{r}$ there is a local smooth diffeomorphism that straightens $B_{1/2}(y_0)\cup\tilde{\Gamma}_{1}^{r}$ to a flat boundary. Since $W=0$ on $\tilde{\Gamma}_{1}^{r}$, we have the vanishing boundary value after flatting. So that we can differentiate the equation in the horizon direction and then employ the $W^{2,p}$ estimates for elliptic equations with partially vanishing boundary value (see \cite[Theorem 9.13]{GT}) to obtain high order derivatives estimates: for $k\ge0$, it holds
	\begin{equation}\label{wkp1}
		\|W\|_{W^{k+2,p}(Q_{1/2})}\lesssim \| W\|_{L^{p}(Q_{1})}+\sum_{m=0}^{k}\|\nabla^{m}F_{l}\|_{L^{\infty}(Q_{1})}.
	\end{equation}
	By Sobolev embedding theorem $W^{s+1,p}(Q_{1/2})\hookrightarrow W^{s,\infty}({Q}_{1/2})$ and $W^{1,2}(Q_{1})\hookrightarrow L^p(Q_{1})$ for some $p>n$, $n=2,3$, together with \eqref{wkp1} and the Poincar\'e inequality, we obtain that, for $1 \leq s \leq l$,
	\begin{equation}\label{wkp3}
		\|\nabla^{s}W\|_{L^{\infty}(Q_{1/2})}\lesssim \|W\|_{W^{s+1,p}(Q_{1/2})}  \lesssim \|\nabla W\|_{L^{2}(Q_{1})}+\sum_{m=0}^{s-1}\|\nabla^{m}F_{l}\|_{L^{\infty}(Q_{1})}. 
	\end{equation}
	Rescaling back to the domain $R(z')$, it follows from \eqref{wkp3} that \eqref{estmain} holds. This completes the proof of Lemma \ref{w1infty}. 
\end{proof}


From \eqref{estmain}, the estimate of $\|\nabla^{s}w_{l}\|_{L^{\infty}(R(z'))}$ requires only estimates of the local $L^{2}$-norm $\|\nabla w_{l}\|_{L^{2}(R(z'))}$ and the $L^{\infty}$-norm $\|\nabla^{m}f_{l}\|_{L^{\infty}(R(z'))}$. 
First, we have the following local energy estimates by using the iteration technique.
\begin{lemma}\label{jbnl}
Under the conditions of Proposition \ref{thm:v1est}, the following estimate holds for sufficiently small $0<\varepsilon<1/4$: 
	\begin{align*}
		\int_{R(z')}|\nabla w_{l}|^2 \diff x \lesssim \delta(z')^{2l+n-2}\quad\mbox{for}\,\,\, (z',z_n)\in \Omega_{1/4}.
	\end{align*}
\end{lemma}
\begin{proof}
By the maximum principle, we have $|v_{1}|\lesssim 1$. In view of the definition $\tilde v_l(x',x_n) := \sum_{k=1}^{l} {\bar v}_k(x',x_n) $ and the estimate in \eqref{scfkdgj}, it follows that $\tilde{v}_{l}$ is bounded and therefore $w_{l}=v_1-\tilde v_l$ is also bounded, i.e., $$\|w_{l}\|_{L^{\infty}(\Omega_{1/2})}\lesssim 1 . $$

Let $\xi$ be a smooth function satisfying the following four conditions: 
(i) $\xi(x')=1$ if $|x'|<1/4$, (ii) $\xi(x')=0$ if $|x'|>1/2$, (iii) $0\leqslant\xi(x')\leqslant1$ if $1/4\le |x'|\le 1/2$, (iv) $|\nabla_{x'}\xi(x')|\lesssim1$. 
Testing equation \eqref{equ_w} with $\xi^2 w_l$ leads to the following result: 
\begin{equation*}
     \int_{\Omega_{1/2}}|\nabla w_l|^2\xi^2\leq C\int_{\Omega_{1/2}}w_l(f_l\xi^2+|\nabla_{x'}\xi|^2)\leq C\|w_{l}\|_{L^{\infty}(\Omega_{1/2})}\int_{\Omega_{1/2}}\delta(x')^{-1}\leq C,
\end{equation*}
which implies that
\begin{equation}\label{nw_bdd}
    \int_{\Omega_{1/4}}|\nabla w_l|^2\leq C.
\end{equation}

We adapt the iteration technique developed in \cite{Li20}. For $|z'|\leq 1/4$ and $0<r<s<1/4$, let $\eta$ be a smooth cutoff function satisfying $\eta(x')=1$ if $|x'-z'|<r$, $\eta(x')=0$ if $|x'-z'|>s$, $0\leq\eta(x')\leq1$ if $t\leq|x'-z'|\leq\,s$, and $|\nabla_{x'}\eta(x')|\leq\frac{2}{s-r}$. Multiplying the equation in \eqref{equ_w} by $w\eta^{2}$ and integrating by parts leads  to the following Caccioppoli's type inequality
 	\begin{equation}\label{FsFt11}
		\int_{R(r,z')}|\nabla{w}_{l}|^{2}\leq\,\frac{C}{(s-t)^{2}}\int_{R(s,z')}|w_{l}|^{2}
 		+C(s-t)^{2}\int_{R(s,z')}\left|f_{l}\right|^{2},
 	\end{equation}
    where the narrow region $R(r,z')$ is defined in \eqref{defraz}.
 	 Since $w_l=0$ on $\Gamma_{1/2}^{\pm}$, by using Poincar\'e inequality, we derive
	\begin{align}\label{energy_w_square}
\int_{R(s,z')}|w_{l}|^{2}\lesssim\,\delta(z')^2\int_{R(s,z')}|\nabla{w}_{l}|^{2},\quad 0\le s\lesssim \delta(z')^{1/2}.
	\end{align}
 By using the estimate of $f_l$ in \eqref{scfkdgj}, we have 
 	\begin{align}\label{zydgj1sc}
 	\int_{R(s,z')}|f_{l}(x',x_n)|^2\lesssim s^{n-1}\delta(z')^{2l-3},\quad 0\le s\lesssim \delta(z')^{1/2}.
 \end{align}
 	Substituting \eqref{energy_w_square} and \eqref{zydgj1sc} into \eqref{FsFt11} and denoting $F(r):=\int_{R(r,z')}|\nabla{w}_{l}|^{2},$
 	we have
 	\begin{equation}\label{tildeF111}
 		F(r)\leq\,\left(\frac{C_{0}\delta(z')}{s-r}\right)^{2}F(s)+C(s-r)^{2}s^{n-1}\delta(z')^{2l-3},
 	\end{equation}
 	where $C_0$ is a fixed positive universal constant.
	
 	Let $k=\big({4C_{0}\delta(z')^{1/2}}\big)^{-1}$ and $r_{i}=\delta(z')+2C_{0}i\,\delta(z')$, $i=0,1,2,\cdots,k$. So, applying \eqref{tildeF111} with $s=r_{i+1}$ and $r=r_{i}$ in \eqref{tildeF111}, we have the following iteration formula: 
 	$$F(r_{i})\leq\,\frac{1}{4}F(r_{i+1})+C\delta(z')^{2l+n-2}(i+1)^{n-1}.$$
 	After $k$ iterations, using \eqref{nw_bdd},
 	\begin{eqnarray*}
		F(r_{0}) \leq \Big(\frac{1}{4}\Big)^{k}F(r_{k})+C\delta(z')^{2l+n-2}\sum_{i=1}^{k}\Big(\frac{1}{4}\Big)^{i-1}(i+1)^l\lesssim \delta(z')^{2l+n-2},
	\end{eqnarray*}
for sufficiently small $\varepsilon$. Thus, Lemma \ref{jbnl} is proved.
\end{proof} 

Using Lemma \ref{w1infty} and Lemma \ref{jbnl}, we reduce estimating higher derivatives of $w_l$ to iteratively improving estimates for $f_l$ and its derivatives, which has been proved in \eqref{scxydest2}. This leads to the result of Lemma \ref{propnew} with $1 \le s \le l$. The case $s=0$ follows directly from the mean value theorem. The proof of Lemma \ref{propnew} is completed. \qed

\section{An example of explicit auxiliary functions for high-order derivatives}\label{app:aux}

In narrow regions, the auxiliary functions from Section \ref{sec:aux} admit more explicit constructions. This yields clearer asymptotic formulas for $\nabla^{l} v_1$, where $v_1$ solves \eqref{equ-v1} and shows the optimality of the estimates obtained in Proposition \ref{thm:v1est}. 

Recalling that the vertical distance between $D_1$ and $D_2$ is $\delta(x')=\phi_1(x')-\phi_2(x').$ Define
\begin{equation}\label{kform}
	k(x):=\frac{x_n-\phi_2(x')}{\delta(x')}-\frac{1}{2}=\frac{x_n}{\delta(x')}-\frac{(\phi_1+\phi_2)(x')}{2\delta(x')},
\end{equation}
and 
\begin{equation*}
	\begin{aligned}
		k(x)^2-\frac14
		&=\frac{1}{\delta(x')^2}\Big(x_n^2-(\phi_1+\phi_2)(x')x_n+\frac14((\phi_1+\phi_2)(x')^2-
		\delta(x')^2)\Big) \\
        &:=\frac{1}{\delta(x')^2}(x_n^2+\mathfrak{h}_{1}(x')x_n+
		\mathfrak{h}_{0}(x') ).
	\end{aligned}
\end{equation*}
One can see that $k(x)=\frac{1}{2}$ on $\Gamma_{1/2}^+$ and $k(x)=-\frac{1}{2}$ on $\Gamma_{1/2}^-$, and so
\begin{equation*}
	k(x)+\frac12=1\,\text{on}\,~\Gamma_{1/2}^+,\quad k(x)+\frac12=0\,~\text{on}\,
	\Gamma_{1/2}^-, \quad \text{and}~~k(x)^2-\frac14=0\,~\text{on}\,\Gamma_{1/2}^+\cup\Gamma_{1/2}^-.
\end{equation*}

Under the assumptions of Proposition \ref{thm:v1est}, for sufficiently small $0 < \varepsilon < 1/2$ and any $l \geq 1$, there exists a polynomial $\tilde{v}_l(x',x_n):=\sum_{k=1}^{l}\bar{v}_k(x',x_n)$ of order $2l-1$ such that
\begin{align}\label{ys1xz}
	|\nabla^{l} \big(v_1(x',x_n)-\tilde{v}_l(x',x_n)\big)|\leq\,C\quad\mbox{in}~\Omega_{1/4},
\end{align}
where 
\[\bar{v}_{1}(x',x_n):=k(x)+\frac{1}{2}=\frac{x_n-\phi_2(x')}{\delta(x')}.\]
Instead of $\eqref{eq:vk}$, which was defined via the Green function $\eqref{defgreenfunction3d}$ earlier, $\bar{v}_k$ are now defined inductively by
\begin{align}\label{vkdgzxz}
	\bar{v}_k(x',x_n)=\sum_{i=0}^{2k-3}\frac{\mathcal{P}_{k,i}(x')}{\delta(x')^2} x_n^{i}(x_n^2+\mathfrak{h}_{1}(x')x_n+
	\mathfrak{h}_{0}(x') ),\quad k\geq2.
\end{align}
Here $\mathcal{P}_{2,1}(x')=-\frac{1}{6}\delta(x')^2\Delta_{x'}\frac{1}{\delta(x')},$ $\mathcal{P}_{2,0}(x')=\frac{1}{2}\delta(x')^2\Delta_{x'}\frac{\phi_2(x')}{\delta(x')}-\mathfrak{h}_{1}\mathcal{P}_{2,1}(x')$, and for $k\geq3$, $0\leq\,i\leq\,2k-3$,
\begin{align}\label{fz2?xz}
	\mathcal{P}_{k,i}(x')=\,\frac{-1}{(i+1)(i+2)}\delta(x')^2\Delta_{x'}\Big(\frac{\mathcal{P}_{k-1,i-2}(x')}{\delta(x')^2}&+\mathfrak{h}_{1}\frac{\mathcal{P}_{k-1,i-1}(x')}{\delta(x')^2}+\mathfrak{h}_{0}\frac{\mathcal{P}_{k-1,i}(x')}{\delta(x')^2}\Big)\nonumber\\
	&-\mathfrak{h}_{1}	\mathcal{P}_{k,i+1}(x')-\mathfrak{h}_{0}	\mathcal{P}_{k,i+2}(x').
\end{align}
We set $\mathcal{P}_{k,i}(x')\equiv0$ for $i\notin\{1,2,\cdots,2k-3\}$.

For some special domains, the auxiliary functions admit a much simpler expression. Suppose $\phi_1(x')$ and $\phi_2(x')$ are quadratic and symmetric about ${x_n = 0}$ for $|x'| \leq 2R$, with $\phi_1(x') = \frac{\varepsilon}{2} + \frac{1}{2}|x'|^2$ and $\phi_2(x') = -\frac{\varepsilon}{2} - \frac{1}{2}|x'|^2$. The vertical distance between $D_1$ and $D_2$ is then $\delta(x') = \varepsilon + |x'|^2$.

Under the assumptions of Proposition \ref{thm:v1est}, for sufficiently small $0 < \varepsilon < 1/2$ and any $l \geq 1$, there exists a polynomial $\tilde{v}_l(x',x_n):=\sum_{k=1}^{l}\bar{v}_k(x',x_n)$ of order $2l-1$ such that
  \begin{align}\label{ys1}
  |\nabla^{l} \big(v_1(x',x_n)-\tilde{v}_l(x',x_n)\big)|\leq\,C\quad\mbox{in}~\Omega_{1/4},
  \end{align}
  where 
  $\bar{v}_{1}(x',x_n):=\frac{x_n}{\delta(x')}+\frac{1}{2}$, instead of \eqref{eq:vk} defined by Green function \eqref{defgreenfunction3d} before, $\bar{v}_k$ are defined inductively by 
  \begin{align}\label{vkdgz}
  \bar{v}_k(x',x_n)=\sum_{i=1}^{k-1} \mathcal{P}_{k,i}(x')x_n^{2k-2i-1}\Big(x_n^2-\frac{1}{4}\delta(x')^2\Big),\quad k\geq2.
  \end{align}
  Here $\mathcal{P}_{2,1}(x')=\frac{1}{3}\Big(\frac{n-1}{\delta(x')^2}-\frac{4|x'|^2}{\delta(x')^{3}}\Big),$ and for $k\geq3$, $1\leq\,i\leq\,k-1$,
  \begin{align*}
  	\mathcal{P}_{k,i}(x')=& \frac{1}{a_{k,i}(a_{k,i}+1)}\Delta_{x'}\Big(\frac{1}{4}\mathcal{P}_{k-1,i-1}(x')\delta(x')^2\!-\mathcal{P}_{k-1,i}(x')\Big)+\frac{1}{4}\mathcal{P}_{k,i-1}(x')\delta(x')^2 ,
  \end{align*}
  with $a_{k,i}=2(k-i)$. We set $\mathcal{P}_{k,i}(x')\equiv0$ for $i\notin\{1,2,\cdots,k-1\}$. Namely, for instance, 
\begin{align*}
\bar{v}_{2}(x',x_n)&=\frac{x_{n}}{3\delta(x')}\Big((n-1)\delta(x')-4|x'|^2\Big)\Big(\frac{x_{n}^{2}}{\delta(x')^{2}}-\frac{1}{4}\Big),\\
\bar{v}_{3}(x',x_n)&=\bigg[\frac{(n^2-1)}{15}\frac{x_{n}^{3}}{\delta(x')}-\Big(2(n+1)\frac{|x'|^2x_{n}^{2}}{\delta(x')^{2}}+\frac{7n+3}{18}|x'|^2-\frac{n^2-1}{60}\delta(x')\Big)x_n\\
&\hspace{4cm} -\frac{4}{5}\Big(\frac{16x_{n}^{2}}{\delta(x')^{2}}+\frac{41}{9}\Big)\frac{x_{n}|x'|^4}{\delta(x')}\bigg]\bigg(\frac{x_{n}^{2}}{\delta(x')^{2}}-\frac{1}{4}\bigg).
\end{align*}

It is easy to check that for $l=1$, we have $\nabla v_{1}=\nabla\bar{v}_{1}+O(1)$, where
\begin{align*}
	|\bar{v}_1|\leq C,\quad|\partial_{x_i}\bar{v}_1|\leq C\delta(x')^{-\frac12},\,~i=1,2,\dots,n-1,\quad\partial_{x_n}\bar{v}_1=\delta(x')^{-1}.
\end{align*}
So that $|\nabla v_1|\leq\,C\delta(x')^{-1}$. By virtue of \eqref{ys1}, 
$$|\nabla v_1|\geq|\nabla \bar{v}_1|-C\geq|\partial_{x_n} \bar{v}_1|-C\ge(C\delta(x'))^{-1},$$ 
implying the optimality for $l=1$. In fact, all the estimates in Proposition \ref{thm:v1est} are optimal.

For $l=2$, we have $\nabla^{2}v_{1}=\nabla^{2}(\bar{v}_1+\bar{v}_2)+O(1)$, where 
$$
|\bar{v}_2|\leq C\delta(x'), \; 
|\nabla\tilde{v}_{2}|=|\nabla(\bar{v}_1+\bar{v}_2)|\leq  C\delta(x')^{-1}, \;
|\nabla^{2}\tilde{v}_{2}|=\Big|\nabla^2\big(\bar{v}_1+\bar{v}_2\big)\Big|\leq C\delta(x')^{-3/2}.
$$
So that $|\nabla^{2} v_1|\leq\,C\delta(x')^{-3/2}$.
This leading terms in $\nabla^{2}\tilde{v}_{2}$ are $\partial_{x_j x_n} \bar{v}_1 = -\frac{2x_j}{\delta(x')^2}$, $j=1,2,\dots,n-1.$ By \eqref{ys1}, 
$$|\nabla^{2} v_1|\geq|\nabla^2 \tilde{v}_{2}|-C\ge\frac{2|x'|}{\delta(x')^2}-C\ge\frac{1}{C|x'|^3}, \quad\text{if } |x'|\ge\sqrt{\varepsilon}.$$

For $l=3$,  we have $\nabla^{3}v_{1}=\nabla^{3}\tilde{v}_3+O(1)=\nabla^{3}(\bar{v}_1+\bar{v}_2+\bar{v}_3)+O(1)$, where
$$|\bar{v}_3|\leq C\delta(x')^{2},\quad|\nabla\tilde{v}_{3}|\leq  C\delta(x')^{-1},\quad|\nabla^2\tilde{v}_{3}|\leq  C\delta(x')^{-3/2}, $$
and 
\begin{align*}
\Big|\nabla^3 \tilde{v}_{3}\Big|&\leq|\nabla^3\bar{v}_1|+|\nabla^3\bar{v}_2| +|\nabla^3\bar{v}_3| \leq C\delta(x')^{-2}+C\delta(x')^{-2}+C\delta(x')^{-1}\leq C\delta(x')^{-2}.
\end{align*}
Leading $\delta(x')^{-2}$-order terms include $\partial_{x_ix_jx_n}\bar{v}_1=\frac{8x_ix_j}{\delta(x')^3}-\frac{2\delta_{ij}}{\delta(x')^2},$ $i,j=1,2,\dots,n-1,$ and $\partial^{3}_{x_n}\bar{v}_2(x',x_n)=\frac{1}{3\delta(x')^3}\Big((n-1)\delta(x')-4|x'|^2\Big)$. Thus, 
$$|\nabla^3\tilde{v}_{3}|\ge \frac{1}{C}\sum_{i=1}^{n-1}|
\partial_{x_ix_ix_n}\bar{v}_1|\ge\frac{8|x'|^2-2(n-1)\delta(x')}{C\delta(x')^3}\ge \frac{1}{C|x'|^4},$$ if $|x'|\ge2\sqrt{\varepsilon}$, implying the optimality of the upper bound.

By induction, auxiliary functions for $l \geq 4$ similarly isolate all singular terms in $\nabla^l v_1$ up to $O(1)$.

\bibliographystyle{plain}
\bibliography{references}

\end{document}